\documentclass[a4paper, 10pt]{article}
\usepackage{arXiv_style}

\usepackage[english]{babel}
\usepackage[utf8]{inputenc} % allow utf-8 input
\usepackage[T1]{fontenc}    % use 8-bit T1 fonts
\usepackage{hyperref}       % hyperlinks
\usepackage{url}            % simple URL typesetting
\usepackage{booktabs}       % professional-quality tables
\usepackage{amsfonts}       % blackboard math symbols
\usepackage{nicefrac}       % compact symbols for 1/2, etc.
\usepackage{microtype}      % microtypography
\usepackage{xcolor}         % colors
\usepackage{amssymb, amsmath, latexsym}
\usepackage{algorithm}
\usepackage[font=small]{caption}
\usepackage{subcaption}
\usepackage{algpseudocode}
\usepackage{bm}
\usepackage{color}
\usepackage{graphicx}
\usepackage{enumitem}
\usepackage{mathtools}
\usepackage{appendix}
\usepackage{cleveref}
\usepackage[makeroom]{cancel}
\usepackage{dsfont}
\usepackage{amsthm}

\usepackage{indentfirst}
\allowdisplaybreaks

\newtheorem{theorem}{Theorem}
\newtheorem{proposition}{Proposition}
\newtheorem{lemma}{Lemma}
\newtheorem{corollary}{Corollary}
\newtheorem{definition}{Definition}
\newtheorem{assumption}{Assumption}
\newtheorem{remark}{Remark}

\title{Accelerated Stochastic ExtraGradient: Mixing Hessian and Gradient Similarity to Reduce Communication in Distributed and Federated Learning}

\author{
  \bf{Dmitry Bylinkin}
 \institute{MIPT}
  \email{dbylinkin69@gmail.com} 
  \and 
 \bf{ Kirill Degtyarev}
   \institute{MIPT}
  \email{degtiarev.kd@phystech.su}
  \and
\bf{ Aleksandr Beznosikov}
 \institute{MIPT, ISP RAS,
Innopolis University}
  \email{beznosikov.an@phystech.su}
     }

\newcommand*{\E}{\mathbb{E}}

\newcommand{\A}{\bar{A}_{\theta}^k}
\newcommand{\nA}{\nabla\bar{A}_{\theta}^k}
\newcommand{\R}{\mathbb{R}^d}
\newcommand{\x}{x_f^{k+1}}

\begin{document}
\maketitle

\begin{abstract}
    Modern realities and trends in learning require more and more generalization ability of models, which leads to an increase in both models and training sample size. It is already difficult to solve such tasks in a single device mode. This is the reason why distributed and federated learning approaches are becoming more popular every day.
    Distributed computing involves communication between devices, which requires solving two key problems: efficiency and privacy. One of the most well-known approaches to combat communication costs is to exploit the similarity of local data. Both Hessian similarity and homogeneous gradients have been studied in the literature, but separately. In this paper, we combine both of these assumptions in analyzing a new method that incorporates the ideas of using data similarity and clients sampling. Moreover, to address privacy concerns, we apply the technique of additional noise and analyze its impact on the convergence of the proposed method. The theory is confirmed by training on real datasets.
\end{abstract}

\section{Introduction}
Empirical risk minimization is a key concept in supervised machine learning \citep{intro_erm}. This paradigm is used to train a wide range of models, such as linear and logistic regression \citep{intro_regression}, support vector machines \citep{intro_svm}, tree-based models \citep{intro_trees}, neural networks \citep{intro_nn}, and many more. In essence, this is an optimization of the parameters based on the average losses in data. By minimizing the empirical risk, the model aims to generalize well to unseen samples and make accurate predictions. In order to enhance the overall predictive performance and effectively tackle advanced tasks, it is imperative to train on large datasets. By leveraging such ones during the training process, we aim to maximize the generalization capabilities of models and equip them with the robustness and adaptability required to excel in challenging real-world scenarios \citep{intro_robust}. The richness and variety of samples in large datasets provides the depth needed to learn complex patterns, relationships, and nuances, enabling them to make accurate and reliable predictions across a wide range of inputs and conditions. As datasets grow in size and complexity, traditional optimization algorithms may struggle to converge in a reasonable amount of time. This increases the interest in developing distributed methods \citep{intro_data_distrib}. Formally, the objective function in the minimization problem is a finite sum over nodes, each one containing its own possibly private dataset. In the most general form, we have:
\begin{equation}\label{prob_form}
    \min_{x\in\R}\left[r(x) = \frac{1}{M}\sum_{m=1}^Mr_m(x) = \frac{1}{M}\sum_{m=1}^M\left( \frac{1}{n_m} \sum_{j=1}^{n_m} \ell(x, y_j^m)\right)\right],
\end{equation}
where $M$ refers to number of nodes/clients/machines/devices, $n_m$ is the size of the $m$-th local dataset, $x$ is the vector representation of a model using $d$ features, $y_j^m$ is the $j$-th data point on the $m$-th node and $\ell$ is the loss function. We consider the most computationally powerful device that can communicate with all others to be a server. Without loss of generality, we assume that $r_1$ is responsible for the server, and the rest $r_m$ for other nodes.

This formulation of the problem entails a number of challenges. 
In particular, transferring information between devices incurs communication overhead \citep{intro_communication}. This can result in increased resource utilization, which affects the overall efficiency of the training process. 
There are different ways to deal with communication costs \citep{mcmahan2017communication,koloskova2020unified,alistarh2017qsgd,beznosikov2023biased}. One of the key approaches is the exploiting the similarity of local data and hence the corresponding losses. One way to formally express it is to use Hessian similarity \citep{intro_similarity}:
\begin{equation}
 \| \nabla^2 r_m (x) - \nabla^2 r(x)\| \leq \delta,
\end{equation}
where $\delta>0$ measures the degree of similarity between the corresponding Hessian matrices. The goal is to build methods with communication complexity depends on $\delta$ because it is known that $\delta\thicksim\nicefrac{1}{\sqrt{n}}$ for non-quadratic losses and $\delta\thicksim\nicefrac{1}{n}$ for quadratic ones \citep{mirror_2}. Here $n$ is the size of the training set. Hence, the more data on nodes, the more statistically "similar" the losses are and the less communication is needed. The basic method in this case is \texttt{Mirror Descent} with the appropriate Bregman divergence \citep{mirror_1,mirror_2}:
\begin{equation*}
x_{k+1} = \arg\min_{x \in \R}\left( r_1 (x) + \langle \nabla r(x_k) - \nabla r_1 (x_k), x \rangle + \frac{\delta}{2}\|x - x_k \|^2 \right).
\end{equation*}
The computation of $\nabla r$ requires communication with clients. Our goal is to reduce it. The key idea is that by offloading nodes, we move the most computationally demanding work to the server. Each node performs one computation per iteration (to collect $\nabla r$ on the server), while server solves the local subproblem and accesses the local oracle $\nabla r_1$ much more frequently. While this approach provides improvements in communication complexity over simple distributed version of \texttt{Gradient Descent}, there is room for improvement in various aspects.  In particular, the problem of constructing an optimal algorithm has been open for a long time \citep{Sim_ohad,Yos_sim,Sim_yan,Sim_sonata}. This has recently been solved by the introduction of \texttt{Accelerated ExtraGradient} \citep{https://doi.org/10.48550/arxiv.2205.15136}.

Despite the fact that there is an optimal algorithm for communication, there are still open challenges: since each dataset is stored on its own node, it is expensive to compute the full gradient. Moreover, by transmitting it, we lose the ability to keep procedure private \citep{intro_privacy}, since in this case it is possible to recover local data. One of the best known cheap ways to protect privacy is to add noise to communicated information \citep{abadi2016deep}. To summarize the above, in this paper, we are interested in stochastic modifications of \texttt{Accelerated ExtraGradient}  \citep{https://doi.org/10.48550/arxiv.2205.15136}, which would be particularly resistant to the imposition of noise on the transmitted packages.

\section{Related Works}

Since our goal is to investigate stochastic versions of algorithms that exploit data similarity (in particular, to speed up the computation of local gradients and to protect privacy), we give a brief literature review of related areas.

\subsection{Distributed Methods under Similarity}
% \subsection{Hessian Similarity}
In $2014$ the scalable $\texttt{DANE}$ algorithm was introduced \citep{Sim_ohad}. It was a novel Newton-type method that used $\mathcal{O}\left( 
\left(\nicefrac{\delta}{\mu}\right)^2\log\nicefrac{1}{\varepsilon} \right)$ communication rounds for $\mu$-strongly convex quadratic losses by taking a step appropriate to the geometry of the problem. This work can be considered as the first that started to apply the Hessian similarity in the distributed setting.  A year later $\Omega\left( \sqrt{\nicefrac{\delta}{\mu}}\log\nicefrac{1}{\varepsilon} \right)$ was found to be the lower bound for communication rounds under the Hessian similarity \citep{Yos_sim}. Since then, extensive work has been done to achieve the outlined optimal complexity. The idea of $\texttt{DANE}$ was developed by adding momentum acceleration to the step, thereby reducing the complexity for quadratic loss to $\mathcal{O}\left( \sqrt{\nicefrac{\delta}{\mu}}\log\nicefrac{1}{\varepsilon} \right)$ \citep{Sim_yan}. Later, \texttt{Accelerated SONATA} achieved $\mathcal{O}\left( \sqrt{\nicefrac{\delta}{\mu}}\log^2\nicefrac{1}{\varepsilon} \right)$ for non-quadratic losses by the \texttt{Catalyst} acceleration \citep{intro_similarity}. Finally, there is \texttt{Accelerated ExtraGradient}, which achieves $\mathcal{O}\left( \sqrt{\nicefrac{\delta}{\mu}}\log\nicefrac{1}{\varepsilon} \right)$ communication complexity using ideas of the Nesterov's acceleration, sliding and ExtraGradient \citep{https://doi.org/10.48550/arxiv.2205.15136}.

% \subsection{Gradient Similarity}
Other ways to introduce the definition of similarity are also found in the literature. In \citep{khaled2022tighter}, the authors study the local technique in the distributed and federated optimization and  $\sigma_{sim}^2=\frac{1}{M}\sum_{m=1}^M\|\nabla r_m(x_*)\|^2$ is taken as measure of gradient similarity. Here $x_*$ is the solution to the problem (\ref{prob_form}). This kind of data heterogeneity is also examine in \citep{woodworth2022minibatch}. The case of gradient similarity of the form $\| \nabla r_m (x) - \nabla r(x) \| \leq \delta$ is discussed in \citep{gorbunov2020local}. Moreover, several papers use gradient similarity in the analysis of distributed saddle point problems \citep{spp_identical,beznosikov2021distributed,beznosikov2022compression,beznosikov2023decentralized}.

\subsection{Privacy Preserving}
A brief overview of the issue of privacy in modern optimization is given in \citep{Priv_rew}. In this paper, we are interested in differential privacy. In essence, differential privacy guarantees that results of a data analysis process remain largely unaffected by the addition or removal of any single individual's data. This is achieved by introducing randomness or noise into the data in a controlled manner, such that the statistical properties of the data remain intact while protecting the privacy of individual samples \citep{priv_easy}. Differential privacy involves randomized perturbations to the information being sent between nodes and server \citep{Priv_dif}. In this case, it is important to design an algorithm that is resistant to transformations.

\subsection{Stochastic Optimization}
The simplest stochastic methods, e.g. \texttt{SGD}, face a common problem: the approximation of the gradient does not tend to zero when searching for the optimum \citep{Stoch_rew}. The solution to this issue appeared in 2014 with invention of variance reduction technique. There are two main approaches to its implementation: \texttt{SAGA} \citep{saga} and \texttt{SVRG} \citep{svrg}. The first one maintains a gradient history and the second introduces a reference point at which the full gradient is computed, and updates it infrequently. The methods listed above do not use acceleration. 
% In this case, the embedding of variance reduction has no effect on linear term in convergence estimates. 
In recent years, a number of accelerated variance reduction methods have emerged \citep{nguyen2017sarah,li2021page,kovalev2020don,katyusha}.

The idea of stochastic methods with variance reduction carries over to distributed setup as well. In particular, stochasticity helps to break through lower bounds under the Hessian similarity condition. For example, methods that use variance reduction to implement client sampling are constructed. \texttt{Catalyzed SVRP} enjoys $\mathcal{O}\left(\left(  1+M^{-\nicefrac{1}{4}}\cdot\sqrt{\nicefrac{\delta}{\mu}}\right)\log^2\nicefrac{1}{\varepsilon}\right)$ complexity in terms of the amount of server-node communication \citep{Sim_svrp}. This method combines stochastic proximal point evaluation, client sampling, variance reduction and acceleration via the Catalyst framework \citep{catalyst_framework}. However, it requires strong convexity of each $r_m$. It is possible to move to weaker assumptions on local functions. There is \texttt{AccSVRS} that achieves $\mathcal{O}\left(\left(1+M^{-\nicefrac{1}{4}}\cdot\sqrt{\nicefrac{\delta}{\mu}}\right)\log\nicefrac{1}{\varepsilon}\right)$ communication complexity \citep{svrs}. In addition, variance reduction turns out to be useful to design methods with compression. There is \texttt{Three Pillars Algorithm}, which achieves $\mathcal{O}\left( \left( 
1+ M^{-\nicefrac{1}{2}}\cdot\nicefrac{\delta}{\mu} \right)\log\nicefrac{1}{\varepsilon} \right)$ for variational inequalities \citep{tpa}.

The weakness of the accelerated variance reduction methods is that they require a small inversely proportional to $\sqrt{M}$ step size \citep{tpa,Sim_svrp,svrs}. Thus, one has to “pay” for the convergence to an arbitrary $\varepsilon$-solution by increasing the iteration complexity of the method. However, there is the already mentioned gain in terms of the number of server-node communications. Meanwhile, accelerated versions of the classic \texttt{SGD} do not have this flaw \citep{vaswani2020adaptive}. In some regimes, the accuracy that can be achieved without the use of variance reduction may be sufficient, including for distributed problems under similarity conditions. 

\section{Our Contribution}
$\bullet$ \textbf{New method.} We propose a new method based on the deterministic \texttt{Accelerated ExtraGradient (AEG)} \citep{https://doi.org/10.48550/arxiv.2205.15136}. Unlike \texttt{AEG}, we add the ability to sample computational nodes. This saves a significant amount of runtime. As mentioned above, sampling was already implemented in \texttt{Three Pillars Algorithm} \citep{tpa}, \texttt{Catalyzed SVRP} \citep{Sim_svrp} and \texttt{AccSVRS} \citep{svrs}. However, this was done using \texttt{SVRG}-like approaches. We rely instead on \texttt{SGD}, which does not require decreasing the step size as the number of computational nodes increases.

$\bullet$ \textbf{Privacy via noise.} We consider stochasticity not only in the choice of nodes, but also in the additive noise in the transmission of information from client to server. This makes it much harder to steal private data in the case of an attack on the algorithm.

$\bullet$ \textbf{Theoretical analysis.} We show that the Hessian similarity implies gradient one with an additional term. Therefore, it is possible to estimate the variance associated with sampling nodes more optimistically. Thus, in some cases \texttt{ASEG} converges quite accurately before getting "stuck" near the solution.

$\bullet$ \textbf{Analysis of subproblem}. We consider different approaches to solving a subproblem that occurs on a server. Our analysis allows us to estimate the number of iterations required to ensure convergence of the method.

$\bullet$ \textbf{Experiments.} We validate the constructed theory with numerical experiments on several real datasets. 

\section{Stochastic Version of Accelerated ExtraGradient}

\begin{algorithm}[t]
	\caption{\texttt{Accelerated Stochastic ExtraGradient}
        \label{st_ae:alg}(\texttt{ASEG})}
	\begin{algorithmic}[1]
		\State {\bf Input:} $x^0=x_f^0 \in \mathbb{R}^d$
		\State {\bf Parameters:} $\tau \in (0,1]$, $\eta,\theta,\alpha > 0, N \in \{1,2,\ldots\}, B \in \mathbb{N}$
		\For{$k=0,1,2,\ldots, N-1$}:
            \For{server}:
			\State $x_g^k = \tau x^k + (1-\tau) x_f^k$\label{st_ae:line:1}
                \State Generate set $|I|=B$ numbers uniformly from $[2,...,n]$ \label{generating_first}
                \State Send to each device $m$ one bit $b_k^m$: $1$ if $m\in I$, 0 otherwise
            \EndFor
            \For{each device $m$ in parallel}:
                \If{$b_k^m=1$}
                    \State Send $\nabla r_m(x_g^k, \xi_m^k)$ to server
                \EndIf
            \EndFor
            \For{server}:
                \State $s_k=\frac{1}{B}\sum_{m\in I}(\nabla r_m(x_g^k, \xi_m^k)-\nabla r_1(x_g^k))$\label{sampling_first}
    		\State $x_f^{k+1} \approx \arg\min_{x \in \mathbb{R}^d}\left[ \A(x) \coloneqq \langle s_k,x - x_g^k\rangle + \frac{1}{2\theta}\lVert x - x_g^k \rVert^2 + r_1(x)\right]$\label{st_ae:line:2}
                \State Generate set $|I|=B$ numbers uniformly from $[1,...,n]$\label{generating_second}
                \State Send to each device $m$ one bit $c_k^m$: $1$ if $m\in I$, $0$ otherwise
            \EndFor
            \For{each device $m$ in parallel}:
                \If{$c_k^m=1$}
                    \State Send $\nabla r_m(x_f^{k+1}, \xi_m^{k+\nicefrac{1}{2}})$ to server
                \EndIf
            \EndFor
            \For{server}:
                \State $t_k = \frac{1}{B}\sum_{m\in I}\nabla r_m(x_f^{k+1}, \xi_m^{k+\nicefrac{1}{2}})$\label{sampling_second}
			\State $x^{k+1} = x^k + \eta\alpha (x_f^{k+1}  - x^k)- \eta t_k$\label{st_ae:line:3}
            \EndFor
		\EndFor
		\State {\bf Output:} $x^N$
	\end{algorithmic}
\end{algorithm}

Algorithm \ref{st_ae:alg} (\texttt{ASEG}) is the modification of \texttt{AEG} \citep{https://doi.org/10.48550/arxiv.2205.15136}. In one iteration, server communicates with clients twice. In the first round, random machines are selected (Line \ref{generating_first}). They perform a computation and send the gradient to the server (Line \ref{sampling_first}). Then the server solves the local subproblem at Line \ref{st_ae:line:2} and the solution is used to compute the stochastic (extra)gradient at the second round of communication (Line \ref{generating_second} and Line \ref{sampling_second}). This is followed by a gradient step with momentum at Line \ref{st_ae:line:3}.

Let us briefly summarize the main differences with \texttt{AEG}. Line \ref{st_ae:line:2} of Algorithm \ref{st_ae:alg} uses a gradient over the local nodes' batches, rather than all of them, as is done in the baseline method. Moreover, the case where local gradients are noisy is supported. This is needed for privacy purposes. Similarly in Line \ref{st_ae:line:3}. Note that we consider the most general case where different devices are involved in solving the subproblem and performing the gradient step.

\texttt{ASEG} requires $\Theta(2B)$ of communication rounds (Line \ref{sampling_first} and Line \ref{sampling_second}), where $B$ is the number of nodes involved in the iteration. Unlike our method, \texttt{AEG} does $\Theta(2M-2)$ per iteration. \texttt{AccSVRS} does $\Theta(2M)$ on average. Thus, in experiments, either the difference between the methods is invisible, or \texttt{AEG} loses. \texttt{ASEG} deals with a local gradient $\nabla r_m(x, \xi)$, which can have the structure:
\begin{align*}
    \nabla r_m(x,\xi)=\nabla r_m(x)+\xi,
\end{align*}
where, for example, $\xi\in N(0,\sigma^2)$ or $\xi\in U(-c,c)$.  This is the previously mentioned additive noise overlay technique, popular in the federated learning.

\section{Theoretical Analysis}
Before we delve into analysis, let us introduce some definitions and assumptions to build the convergence theory.

In this paper, complexity is understood in the sense of number of server-node communications. This means that we consider how many times a server exchange vectors with clients. In this case, a single contact is counted as one communication.

We work with functions of a class widely used in the literature. 
\begin{definition}\label{def:convexity}
We say that the function $f(x)\colon \mathbb{R}^d \rightarrow \mathbb{R}$ is $\mu$-strongly convex on $\mathbb{R}^d$, if
    \begin{align*}
        f(x)\geq f(y)+\langle\nabla f(y),x-y\rangle + \frac{\mu}{2}\|x-y\|^2,\quad \forall x,y\in\R.
    \end{align*}
\end{definition}
If $\mu=0$, we call $f$ a convex function.

\begin{definition}\label{def:smoothness}
    We say that the function $f(x)\colon \mathbb{R}^d \rightarrow \mathbb{R}$ is $L$-smooth on $\mathbb{R}^d$, if
    \begin{align*}
        \|\nabla f(x)-\nabla f(y)\|\leq L\|x-y\|,\quad \forall x,y\in\R.
    \end{align*}
\end{definition}

\begin{definition}\label{def:sim}
    We say that $f_m(\cdot)$ is $\delta$-related to $f(\cdot)$, if
    \begin{equation*}
     \| \nabla^2 f_m (x) - \nabla^2 f(x)\| \leq \delta,\quad\forall x\in\R.
    \end{equation*}
\end{definition}
We assume a strongly convex optimization problem (\ref{prob_form}) with possibly non-convex components.
\begin{assumption}\label{ass:1}
    $r$ is $\mu$-strongly convex, $r_1$ is convex and $L_1$-smooth.
\end{assumption}
This assumption does not require convexity of local functions. Indeed, $\delta$-relatedness is a strong enough property to allow them to be non-convex. Within this paper we are interested in the case $\mu<\delta$, since $\mu\ll\delta\ll L$ for large datasets in practice \citep{sun2022distributed}.

\begin{proposition}\label{prop:almost_sim}
    Consider $\delta$-relatedness $r_m(\cdot)$ to $r(\cdot)$ (Definition \ref{def:sim}) for each $m\in[1,M]$. In this case, we have
    \begin{align*}
        \|\nabla r_m(x)-\nabla r(x)\|^2 \leq \sigma_{sim}^2 + 2\delta^2\|x-x_*\|^2, 
    \end{align*}
    where $\sigma_{sim}^2=2\max_{m\in[1,M]}\{\|\nabla r_m(x_*)\|^2\}$ and $x_*$ is the solution of problem (\ref{prob_form}).
\end{proposition}
\begin{proof}
    It is obvious that Definition \ref{def:sim} implies smoothness of $r_m-r$:
    \begin{align*}
        \|\nabla(r_m-r)(x)-\nabla(r_m-r)(x_*)\|\leq\delta\|x-x_*\|, \quad\forall m\in[1,M],x\in\R.
    \end{align*}
    Since $x_*$ is the optimum, $\nabla r(x_*)=0$. It is known that $\|x-y\|^2\leq2\|x\|^2+2\|y\|^2$. We obtain
    \begin{align*}
        \|\nabla r_m(x) - \nabla r(x)\|^2 \leq 2\|\nabla r_m(x_*)\|^2 + 2\delta^2\|x-x_*\|^2.
    \end{align*}
\end{proof}

Thus, from the similarity of the Hessians follows the similarity of the gradients with some correction. Let us introduce a more general assumption on the gradient similarity:
\begin{assumption}\label{ass:sim}
    For all $m\in[1,M]$, Definition \ref{def:sim} is satisfied, and the following inequality holds:
    \begin{align*}
        \|\nabla r_m(x) - \nabla r(x)\|^2 \leq \sigma_{sim}^2 + 2\delta^2\|x-x_*\|^2\cdot\zeta,
    \end{align*}
    where $\zeta\geq0$.
\end{assumption}
Assumption \ref{ass:sim} is a generalization of Proposition \ref{prop:almost_sim}, where the correction term is multiplied by some constant $\zeta$. Formally, it can take any non-negative value, but it is important to analyze only two cases corresponding to different degrees of similarity: $\zeta=0$ and $\zeta=1$. Indeed, any non-negative values of $\zeta$ give asymptotically the same complexity.

A stochastic oracle is available for the function at each node. We impose the standard assumption:
\begin{assumption}\label{ass:2}
    $\forall m\in\{1,...,n\}$ the stochastic oracle $\nabla r_m(x,\xi)$ is unbiased and has bounded variance:
    \begin{align*}
        \E_{\xi}[\nabla r_m(x, \xi)|x] = \nabla r_m(x),\quad 
        \E_{\xi}[\| \nabla r_m(x, \xi)-\nabla r_m(x) \|^2|x]\leq\sigma^2_{noise}, \quad\forall x\in\R.
    \end{align*}
\end{assumption}

\begin{remark}
    Hessian similarity of $r_m$ and $r$ for $m\neq1$ is only used to construct Proposition \ref{prop:almost_sim} and is not needed anywhere else. The next proofs require only that the server expresses the average nature of the data, while the data on the nodes may be heterogeneous.
\end{remark}

We want to obtain a criterion that guarantees convergence of Algorithm \ref{st_ae:alg}. The obvious idea is to try to get one for the new algorithm by analogy with how it was done in the non-stochastic case \citep{https://doi.org/10.48550/arxiv.2205.15136}:

\begin{lemma}\label{st_ae:main_lemma}
Consider Algorithm \ref{st_ae:alg} for the problem (\ref{prob_form}) under Assumptions \ref{ass:1}, \ref{ass:sim} and \ref{ass:2} with the following tuning: 
\begin{equation}\label{st_ae:naive_choice}
    \alpha=\frac{\mu}{3},\quad\eta\leq\frac{1}{3\alpha},\quad\eta\leq\frac{\theta}{3\tau},\quad\theta\leq\frac{1}{3\delta},
	\end{equation}
and let $x_f^{k+1}$ satisfy    
	\begin{equation}\label{st_aux:grad_app}
	\lVert\nabla \bar{A}_\theta^k(x_f^{k+1})\rVert^2 \leq  \frac{9\delta^2}{11}\lVert x_g^k- \arg\min_{x \in \mathbb{R}^d} \bar{A}_\theta^k(x)\rVert^2.
	\end{equation}
Then the following inequality holds:
        \begin{align*}\label{st:ae_rec}
	\mathbb{E} \bigg[ \frac{1}{\eta} & \lVert x^{k+1} - x^* \rVert^2
	+
	\frac{2}{\tau}\left[r(x_f^{k+1}) - r(x^*)\right]\bigg]
        \\
	\leq&
	\left(1-\alpha \eta +\frac{12 \theta\delta^2 \eta }{B}\cdot\zeta\right)  \cdot \mathbb{E}\left[ \frac{1}{\eta} \lVert x^k - x^* \rVert^2\right]+ (1-\tau) \cdot \E\left[\frac{2}{\tau}\left[r(x_f^k) - r(x^*)\right]\right] 
        \\
        &+\E\left[\left(\frac{12\theta\delta^2}{B\tau}-\frac{\mu}{6}\right) \cdot\zeta \left(\|x_f^k-x_*\|^2 + \|\x-x_*\|^2\right)  \right]
        + \frac{4\theta}{B\tau}\sigma^2.
        \end{align*}
\end{lemma}
See the proof in Appendix \ref{ap:proof_th_1}. We first deal with the case under the gradient similarity condition ($\zeta=0$), since it removes two additional constraints on the parameters and makes the analysis simpler.\\
The following theorem is obtained by rewriting the result of Lemma \ref{st_ae:main_lemma} with a finer tuning of the parameters. We take $\mu<\delta$ into account and find that $1-\frac{\sqrt{\mu\theta}}{3}\geq1-\frac{1}{3}$.
\begin{theorem}\label{th_set_1}
    In Lemma \ref{st_ae:main_lemma} consider Assumption \ref{ass:sim} with $\zeta=0$ and adjust the parameters:
    \begin{equation}\label{st_ae:set_1_choice}
        \alpha=\frac{\mu}{3},\quad\eta=\min\left\{\frac{1}{3\alpha}, \frac{\theta}{3\tau}\right\},\quad\tau=\frac{\sqrt{\mu\theta}}{3},\quad\theta\leq\frac{1}{3\delta}.
    \end{equation}
    Then we have:
    \begin{align*}
        \E[\Phi_{k+1}]\leq\left( 1-\frac{\sqrt{\mu\theta}}{3} \right)\E[\Phi_k] + \frac{4\theta\sigma^2}{B},
    \end{align*}
    where $\Phi_k = \frac{\tau}{\eta}\|x^k-x^*\|^2 + 2[r(x_f^k)-r(x^*)]$.
\end{theorem}
The theorem asserts convergence only to some neighborhood of the solution. To guarantee convergence with arbitrary accuracy, a suitable parameter $\theta$ must be chosen. We propose a corollary of Theorem \ref{th_set_1}:

\begin{corollary}\label{st_ae:cor_1}
    Consider the conditions of Theorem \ref{th_set_1}. With appropriate tuning of $\theta$ Algorithm \ref{st_ae:alg} requires
    $$
    \mathcal{O}\left( \frac{B}{M}\sqrt{\frac{\delta}{\mu}}\log\frac{1}{\varepsilon} + \frac{\sigma_{sim}^2+\sigma_{noise}^2}{\mu M\varepsilon} \right) \quad \text{communications}
    $$
    to reach an arbitrary $\varepsilon$-solution.
\end{corollary}
See the proof in Appendix \ref{ap:proof_cor_1}. Note that the linear part of the obtained estimate reproduces the convergence of \texttt{AEG}. However, for example, with $B=1$, an iteration of \texttt{ASEG} costs $\Theta(2)$ communications instead of $\Theta(2M-2)$. Thus, our method requires significantly less communication to achieve the value determined by the sublinear term. Moreover, as mentioned earlier, in the \texttt{Catalyzed SVRP} and \texttt{AccSVRS} estimates, an additional $M^{-\nicefrac{1}{4}}$ multiplier arises due to the use of variance reduction. Comparing $M^{-\nicefrac{1}{4}}\cdot\sqrt{\nicefrac{\delta}{\mu}}$ to $BM^{-1}\cdot\sqrt{\nicefrac{\delta}{\mu}}$, we notice the superiority of \texttt{ASEG} in terms of chosen approach to measure communication complexity.

Note that if the noise of the gradients sent to the server is weak ($\sigma_{noise}=0$) and the functions are absolutely similar in the gradient sense ($\sigma_{sim}=0$), then \texttt{ASEG} has $\nicefrac{M^{\nicefrac{3}{4}}}{B}$ times less complexity in terms of number of communications.

In the case of a convex objective, it is not possible to implement the method under Assumption \ref{ass:sim} with $\zeta\neq0$, since the arising correction terms cannot be eliminated. Under Assumption \ref{ass:sim} with $\zeta=0$, we propose a modification of \texttt{ASEG} for convex objective. In this case, we need
$$
\mathcal{O}\left( \frac{B}{M}\frac{\sqrt{\delta} \lVert x_0-x_*\rVert }{\sqrt{\varepsilon}} + \frac{(\sigma_{sim}^2+\sigma_{noise}^2)\lVert x_0-x_*\rVert^2 }{M\varepsilon^2} \right) \quad 
\text{communications}$$ to reach an arbitrary $\varepsilon$-solution. See Appendix \ref{ap:aseg_conv}.\\

The case $\zeta=1$ is more complex and requires more fine-tuned analysis.
\begin{theorem}\label{th_set_2}
    In Lemma \ref{st_ae:main_lemma} consider Assumption \ref{ass:sim} with $\zeta=1$ and tune parameters:
    \begin{equation}\label{st_ae:choice}
        \alpha=\frac{\mu}{3},\quad\eta=\min\left\{\frac{1}{3\alpha}, \frac{\theta}{3\tau}\right\},\quad\tau=\sqrt{\mu\theta},\quad\theta\leq\min\left\{\frac{1}{3\delta}, \frac{\mu^3B^2}{5184\delta^4}\right\}.
    \end{equation}
    Then we have:
    \begin{align*}
        \E[\Phi_{k+1}]\leq\left( 1-\frac{\sqrt{\mu\theta}}{18} \right)\E[\Phi_k] + \frac{4\theta\sigma^2}{B},
    \end{align*}
    where $\Phi_k = \frac{\tau}{\eta}\|x^k-x^*\|^2 + 2[r(x_f^k)-r(x^*)]$.
\end{theorem}

See the proof in Appendix \ref{ap:th_set2}.
\begin{corollary}\label{st_ae:cor_2}
    Consider the conditions of Theorem \ref{th_set_2}. With appropriate tuning of $\theta$ Algorithm \ref{st_ae:alg} requires 
    $$\mathcal{O}\left( \frac{B}{M}\sqrt{\frac{\delta}{\mu}}\cdot\log\frac{1}{\varepsilon} + \frac{\delta^2}{\mu^2M}\cdot\log\frac{1}{\varepsilon} + \frac{\sigma_{sim}^2+\sigma_{noise}^2}{\mu M\varepsilon} \right) \quad 
    \text{communications}
    $$
    to reach an arbitrary $\varepsilon$-solution.   
\end{corollary}
See the proof in Appendix \ref{ap:proof_cor_1}. It can be seen from the proof that Assumption \ref{ass:sim} with $\zeta=1$ gives worse estimates. Nevertheless, optimal communication complexity can be achieved if we choose large enough $B$, which equalizes two linear terms (at least $72\frac{\delta^{\nicefrac{3}{2}}}{\mu^{\nicefrac{3}{2}}}$). It comes from equality
\begin{align*}
    \frac{1}{3\delta} = \frac{\mu^3B^2}{5184\delta^4}.    
\end{align*}
Thus, even on tasks with a poor ratio of $\delta$ to $\mu$, if the number of available nodes is large enough, \texttt{ASEG} outperforms \texttt{AEG}, \texttt{Catalyzed SVRP} and \texttt{AccSVRS}.

\section{Analysis of the Subprolem}
In this section, we are focused on
\begin{equation}\label{th2:prob}
    \min_{x \in \mathbb{R}^d}\left[\bar{A}_\theta^k(x) \coloneqq \langle s_k,x - x_g^k\rangle + \frac{1}{2\theta}\lVert x - x_g^k \rVert^2 + r_1(x)\right],
\end{equation}
defined in Line \ref{st_ae:line:2} of Algorithm \ref{st_ae:alg}.
Note that the problem (\ref{th2:prob}) is solved on the server and does not affect the communication complexity. Obviously:
\begin{align*}
    \|\nA(x)-\nA(y)\|=\left\|\frac{x-y}{\theta}+\nabla r_1(x)-\nabla r_1(y)\right\| \leq \left(\frac{1}{\theta}+L_1\right)\|x-y\|.
\end{align*}
Moreover:
\begin{align*}
    \nabla^2 \bar{A}_{\theta}^k(x)=\frac{1}{\theta}I+\nabla^2 r_1(x)\succ 0.
\end{align*}
Thus, subproblem is $L_A$-smooth with $L_A=\frac{1}{\theta}+L_1$ and $\frac{1}{\theta}$ strongly convex.

Despite the fact that the problem (\ref{th2:prob}) does not affect the number of communications, we would still like to spend less server time solving it. The obvious solution is to use stochastic approaches.
\subsection{SGD Approach}
 The basic approach is to use \texttt{SGD} for the strongly convex smooth objective. In this case, the convergence rate of the subproblem by argument norm is given by the following theorem \citep{opt_sgd}:
\begin{theorem}
    Let Assumptions \ref{ass:1} and \ref{ass:2} hold for the problem (\ref{th2:prob}). 
    Consider $T$ iterations of \texttt{SGD}
    with step size $\gamma\leq\nicefrac{1}{2L_A}$.
    Then there is the following estimate:
    \begin{align*}
        \E\left[\left\|x_{T}-\arg\min_{x \in \mathbb{R}^d} \bar{A}_\theta^k(x)\right\|^2|s_k\right]\leq \left(1-\frac{\gamma}{\theta}\right)^{T}\left\|x_0-\arg\min_{x \in \mathbb{R}^d} \bar{A}_\theta^k(x)\right\|^2+\theta\gamma\sigma_1^2.
    \end{align*}
\end{theorem}
Using smoothness of the subproblem, we observe:
\begin{align*}
    \E[\|\nA(x_{T})\|^2|s_k]\leq L_A^2\E\left[\left\|x_{T}-\arg\min_{x \in \mathbb{R}^d} \bar{A}_\theta^k(x)\right\|^2|s_k\right].
\end{align*}
Let us formulate
\begin{corollary}
    Under Assumption \ref{ass:1} and \ref{ass:2}, consider the problem (\ref{th2:prob}) and $T$ iterations of \texttt{SGD}. There is the following estimate:
    \begin{align*}
        \E[\|\nA(x_{T})\|^2|s_k]\leq L_A^2\left(1-\frac{\gamma}{\theta}\right)^T\left\|x_0-\arg\min_{x \in \mathbb{R}^d} \bar{A}_\theta^k(x)\right\|^2+L_A^2\theta\gamma\sigma_1^2.
    \end{align*}
\end{corollary}
Our interest is to choose such step that it would be possible to avoid getting stuck in the neighborhood of a subproblem solution. According to \citep{opt_sgd}, we note the following.
\begin{theorem}
    Under Assumptions \ref{ass:1} and \ref{ass:2}, consider the problem (\ref{th2:prob}) to be solved using \texttt{SGD}. There exists decreasing step sizes $\gamma_t\leq\nicefrac{1}{2L_A}$ such that for $T$ iterations we obtain:
    \begin{align*}
        \E[\|\nA(x_{T})\|^2|s_k]\leq& 64\theta L_A^3\left\| x_0-\arg\min_{x \in \mathbb{R}^d} \bar{A}_\theta^k(x) \right\|^2\exp\left\{ -\frac{T}{4L_A\theta} \right\} + \frac{36L_A^2\sigma_1^2\theta}{T}.
    \end{align*}
\end{theorem}
In this way, we can achieve optimal convergence rate of \texttt{SGD}. However, we would like to enjoy linear convergence. The idea is to use variance reduction methods.

\subsection{SVRG Approach}
Another approach to solving the subproblem is to use variance reduction methods, for example, \texttt{SVRG}. Let us take $x_g^k$ as a starting point and look at the first epoch. It is known from \citep{http://proceedings.mlr.press/v108/gorbunov20a/gorbunov20a-supp.pdf}:
\begin{equation*}
    \E\left[\A\left(\frac{1}{J}\sum_{j=1}^Jx_j\right) - \A(x_*)\right] \leq \frac{2(\mu^{-1}+2J\gamma^2L_A)}{J\gamma(1-2\gamma L_A)}\E\left[\A(x_g^k) - \A(x_*)\right],
\end{equation*}
where $J$ is an epoch size. 
$x=\frac{1}{J}\sum_{j=1}^Jx_j$ is the next point after the algorithm finishes current epoch. We can rewrite for $T$-th epoch:
\begin{equation*}
    \E\left[\A\left(\widehat{x}_T\right) - \A(x_*)\right] \leq \left(\frac{2(\mu^{-1}+2J\gamma^2L_A)}{J\gamma(1-2\gamma L_A)}\right)^T\E\left[\A(x_g^k) - \A(x_*)\right].
\end{equation*}
Then we write strong-convexity property:
\begin{equation*}
    \E\left[\A(x_g^k) - \A(x_*)\right] \leq \frac{1}{2\mu}\E[\lVert\nA(x_g^k)\rVert^2].
\end{equation*}
Using $L_A$-smoothness of $\A$, we obtain
\begin{align*}
    \E\left[\left\lVert\nA\left(\widehat{x}_T\right)\right\rVert^2\right] \leq& 2L\E\left[\A\left(\widehat{x}_T\right) - \A(x_*)\right] 
    \\
    \leq& \frac{L_A}{\mu}\left(\frac{2(\mu^{-1}+2J\gamma^2L_A)}{J\gamma(1-2\gamma L_A)}\right)^T\E[\lVert\nA(x_g^k)\rVert^2].
\end{align*}
Thus, for a suitable combination of epoch size and step, we expect to obtain linear convergence of the subproblem, which is consistent with the experimental results.

\subsection{Loopless Approach} 

In this section, we filling in some gaps in subproblem analysis. Decision criterion for $\texttt{ASEG}$ is (\ref{st_aux:grad_app}). Let us assume that we obtained:
\begin{equation}
    \label{criteria_eq:1}
    ||\nabla A_{\theta}^k(x_T)||^2=\mathcal{O}\left(\frac{1}{T^{\alpha}}\right), \quad \alpha>0.
\end{equation}
We use this to obtain a lower bound on $T$. The important question is whether it is possible to take fewer steps so that the criterion necessary for the convergence of the whole algorithm is met. In other words, we want to check whether the loopless version of \texttt{ASEG} turns out to be better than the usual one, as in the case of \texttt{SVRG} and \texttt{KATYUSHA} \citep{kovalev2020don}.

We are interested in comparing deterministic and random choices of $T$.
Let $\xi$ be a random variable describing the number of steps. $\xi$ takes positive discrete values. Let $A$ be the number of steps for the subproblem of \texttt{ASEG}, defined deterministically and satisfying \eqref{st_aux:grad_app}. 
Let us assume that 
\begin{equation}
    \label{assuming_1:eq0}
    \frac{1}{\xi^{\alpha}}\leq \frac{1}{A^{\alpha}}.
\end{equation}
Then from \eqref{criteria_eq:1} we know that
\begin{equation*}
    \label{gradient_ineq_1:eq0}
    \|\nabla A_{\theta}^k(x_\xi)\|^2\leq \|\nabla A_{\theta}^k(x_A)\|^2,
\end{equation*}
where $x_{\xi}$ is the point obtained through $\xi$ solver steps. Obviously, the reverse is also true. In other words, the condition \eqref{assuming_1:eq0} is the criterion for $\xi$ to solve the initial subproblem. At the same time, we want to minimize $\E\xi$ to take as few steps as possible.
\begin{proposition} \label{prop:1}
	$\E[\xi] < A \Rightarrow \E[\frac{1}{\xi^{\alpha}}] > \frac{1}{A^{\alpha}}$.
\end{proposition}

\begin{proof}
Let $p_i$ be the probability to choose $\xi=i$. Then we have
\begin{align*}
    \E\left[\frac{1}{\xi^{\alpha}}\right]=\sum\limits_{i=1}^Tp_i\frac{1}{i^{\alpha}}.
\end{align*}
Since $\alpha>0$, $\phi(i)=\nicefrac{1}{i^{\alpha}}$ is a convex function. Using the Jensen's inequality, we obtain 
\begin{align*}\E\left[\frac{1}{\xi^{\alpha}}\right] \geq \frac{1}{(\E\xi)^{\alpha}}> \frac{1}{A^{\alpha}}.\end{align*}
\end{proof}
Thus, in the worst case, if the expectation of $\xi$ is less than the deterministic number of iterations, the criterion cannot be met and there is no probabilistic method of choosing the number of iterations that is better than the deterministic one.

It is worth noting that the same can be said about a pair of \texttt{SVRG} and \texttt{L-SVRG}. It is obvious that \texttt{SVRG} could be reduced to \texttt{L-SVRG} if the number of iterations before updating the gradient is considered as a random variable. Thus, the epoch $J$ is from the geometric distribution and one can compare which algorithm gives a more accurate solution after performing the inner loop. For \texttt{SVRG}-rate we have \citep{Stoch_rew}:
\begin{align}\label{eq:svrg_conv}
    \E&\left[\A(x_{J(t+1)}) - \A(x_*)\right] \leq
    \left( \frac{1}{\eta \gamma (1 - 2L\gamma)J} + \frac{2L\gamma}{1 - 2L\gamma}\right)\E\left[\A(x_{Jt}) - \A(x_*)\right],
\end{align}
where $\gamma$ is a step and $\eta > 0$ - some constant.
From the Proposition \ref{prop:1} and (\ref{eq:svrg_conv}) we obtain:

\begin{equation*}
\left( \frac{1}{\eta \gamma (1 - 2L\gamma)} \E\left[\frac{1}{J}\right] + \frac{2L\gamma}{1 - 2L\gamma}\right) \leq \left( \frac{1}{\eta \gamma (1 - 2L\gamma)J} + \frac{2L\gamma}{1 - 2L\gamma}\right).
\end{equation*}
Thus, in the worst case, loopless versions converges no faster than usual ones.

\section{Numerical Experiments}
To support theoretical analysis and insights, we solve both logistic regression problem: 
\begin{equation*}
    r(x) = \frac{1}{M}\sum_{m=1}^M \frac{1}{n}\sum_{j=1}^n \ln\left(1+e^{-b^m_j \left\langle x, a^m_j \right\rangle}\right) + \lambda\|x\|^2,
\end{equation*}
and quadratic regression problem:
\begin{equation*}
    r(x) = \frac{1}{M}\sum_{m=1}^M \frac{1}{n}\sum_{j=1}^n \left(\langle x, a^m_j\rangle -b^m_j\right)^2+ \lambda\|x\|^2.
\end{equation*}
Here $x$ refers to parameters of the model, $a_j^m$ is the feature description and $b_j^m$ is the corresponding label. To avoid overfitting, we use ${l_2}$-regularization. The penalty parameter $\lambda$ is set to $\nicefrac{L}{100}$, where $L$ is a smoothness constant (\ref{def:smoothness}). We use \texttt{SVRG} and \texttt{SARAH} as the subtask solvers in Algorithm \ref{st_ae:alg}. In all tests, the original datasets are merged into batches $\left\{ r_b\right\}_{i=1}^{M+N}$.
%Как подробно писать про батчирование и про распределение функций в разных датасетах по функциям для алгоритма
A function $r_1(x)$ is created from part of the acquired ones:
\begin{equation*}
    r_1(x) = \frac{1}{N}\sum_{b=1}^{N} r_b(x),\quad N < M.
\end{equation*}
Due to the Hessian similarity, we have the following expression for the smoothness constant
\begin{equation*}
    \lVert \nabla^2 r(x) -  \nabla^2 r_1(x) \rVert \leq \delta,
\end{equation*}
where
\begin{equation*}
    r(x) = \frac{1}{M}\sum_{i=1}^M r_i(x).
\end{equation*}
Then for the quadratic task we have
\begin{equation*}
    \delta = \left\lVert \frac{1}{N}\sum_{i=1}^N A_i - \frac{1}{M}\sum_{j=1}^M A_j \right\rVert.
\end{equation*}
For the logistic regression, $\delta$ is estimated by enumerating 100 points in the neighborhood of the solution. In both cases, obtained value is multiplied by $\nicefrac{3}{2}$. Experiments are performed on three datasets: \texttt{Mushrooms}, \texttt{W8A, A9A}, \citep{chang2011libsvm}.

$\bullet$ We investigate how the size of the batches affects the convergence of \texttt{ASEG} on tasks with different parameters. Figure \ref{batching} shows the dependence of untuned \texttt{ASEG} convergence on the batch size. In the first experiment the method on $50$ nodes solves the problem with $\mu=0.105$ and $\delta=1.45$. According to Corollary \ref{st_ae:cor_2} one should choose $B>72\cdot(\nicefrac{\delta}{\mu})^{\nicefrac{3}{2}} \approx 3695$ to obtain optimal complexity. That is, for any size of the batch, there is non-optimal rate. In the second experiment, the method on $200$ nodes solves the problem with $\mu=0.06$, $\delta=0.07$. In this case, \texttt{ASEG} has guaranteed optimal complexity for $B>57$ and converges to rather small values of the criterion.
\begin{figure}[h!]
    \begin{center}
    \includegraphics[width = 0.7\textwidth]{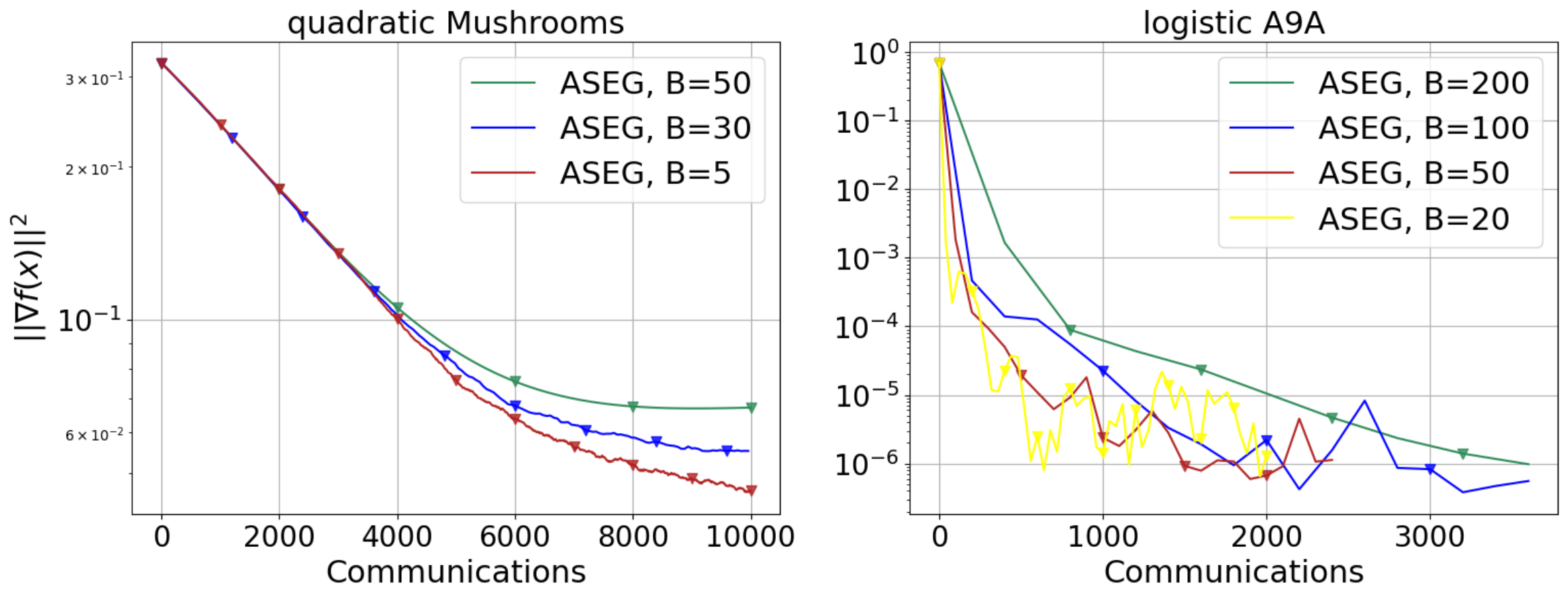}
    \end{center}
    \caption{\texttt{ASEG} with different batch sizes}
    \label{batching}
\end{figure}

$\bullet$ We compare \texttt{ASEG} to \texttt{AccSVRS} (Figure \ref{aseg_vs_svrs_mush} and Figure \ref{aseg_vs_svrs_a9a}). 50 nodes are available for \texttt{Mushrooms} and 200 nodes are available for \texttt{A9A}. If the ratio between $\mu$, $\delta$, and $B$ yields the optimal complexity according to Corollary \ref{st_ae:cor_2}, then due to similarity there is an optimistic estimate on the batching noise provides a significant gain over \texttt{AccSVRS}. Otherwise, it is worth using full batching and obtaining a less appreciable gain.
\begin{figure}[h!]
    \begin{center}
    \includegraphics[width = 0.7\textwidth]{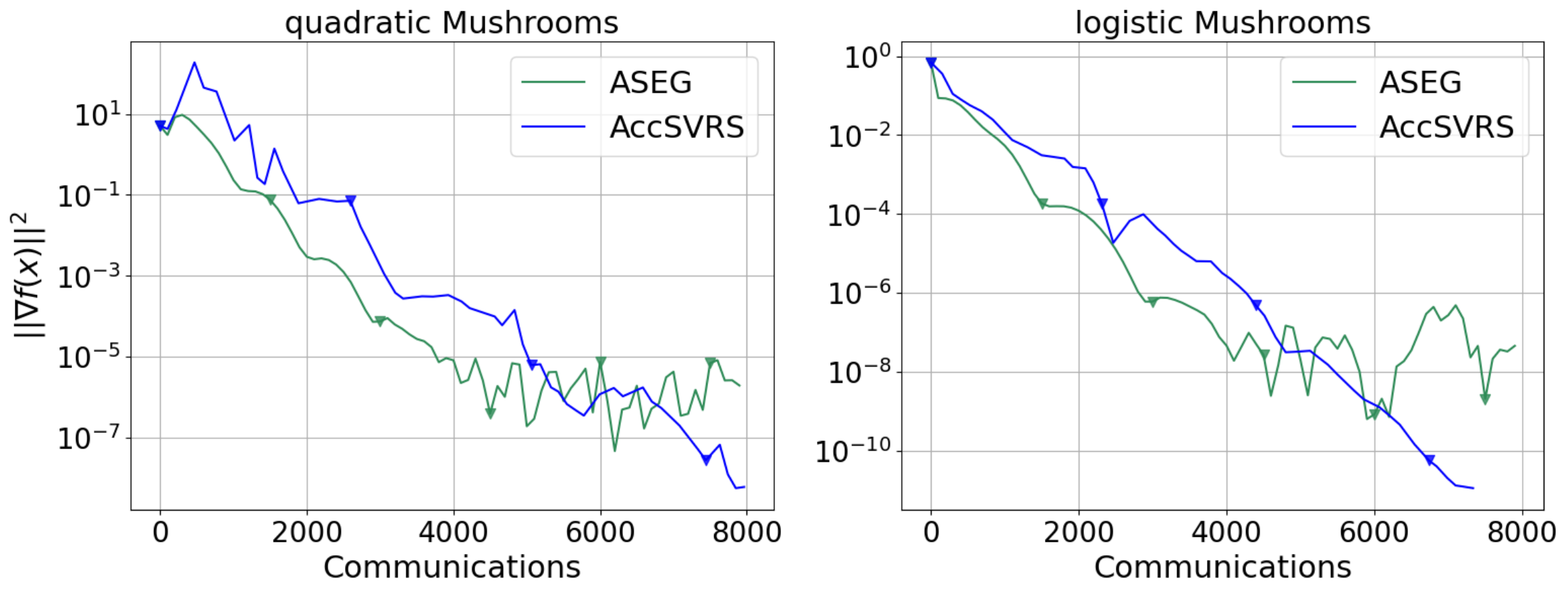}
    \end{center}
    \caption{\texttt{ASEG} vs. \texttt{AccSVRS} on data with $\delta=10.15$ for quadratic task and $\delta=1.45$ for logistic one}
    \label{aseg_vs_svrs_mush}
\end{figure}
\begin{figure}[h!]
    \begin{center}
    \includegraphics[width = 0.7\textwidth]{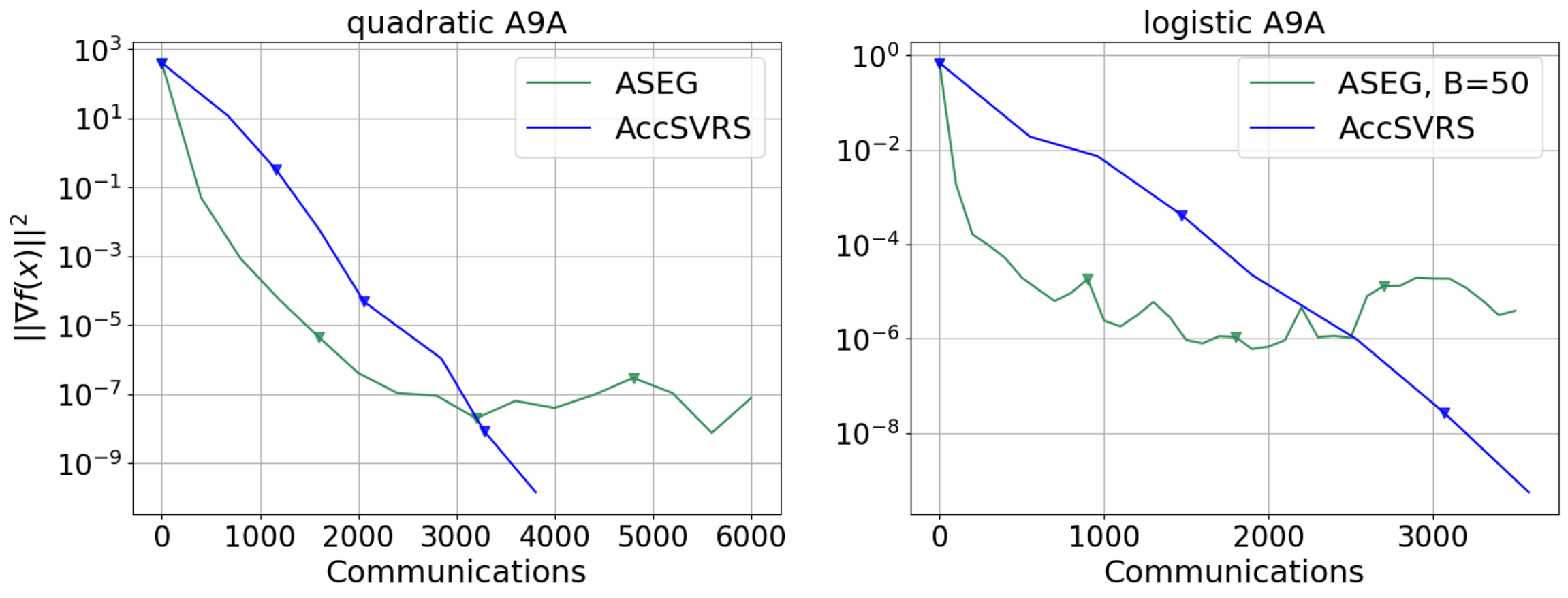}
    \end{center}
    \caption{\texttt{ASEG} vs. \texttt{AccSVRS} on data with $\delta=0.064$ for quadratic task and $\delta=0.061$ for logistic one}
    \label{aseg_vs_svrs_a9a}
\end{figure}

$\bullet$ In \texttt{ASEG} stability experiments (Figure \ref{fig:stability_quadratic_svrg}, Figure \ref{fig:stability_logistic_svrg}, Figure \ref{fig:stability_quadratic_sarah}, Figure \ref{fig:stability_logistic_sarah}) it is tested how white noise in each $\nabla r_m(x)$ affects the convergence. It can be seen that increasing the randomness does not worsen the convergence. This can be explained by the fact that the convergence is mainly determined by accuracy of chosen solver.

\begin{figure}[h!] % <---
   \begin{subfigure}{0.32\textwidth}
       \centering
       \includegraphics[width=\linewidth]{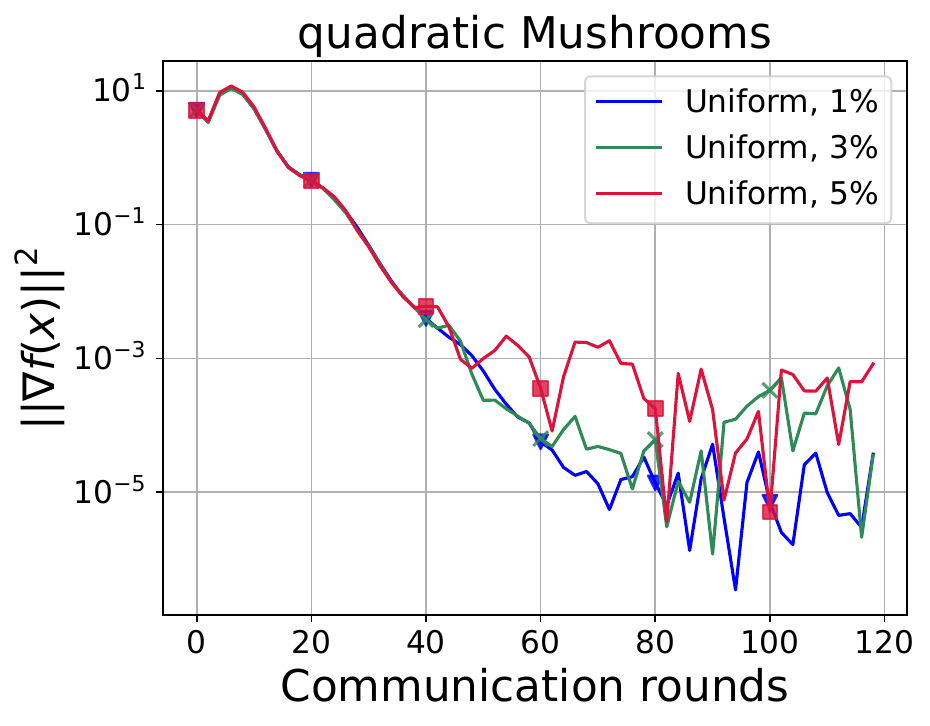}
   \end{subfigure}
   \begin{subfigure}{0.32\textwidth}
        \centering
       \includegraphics[width=\linewidth]{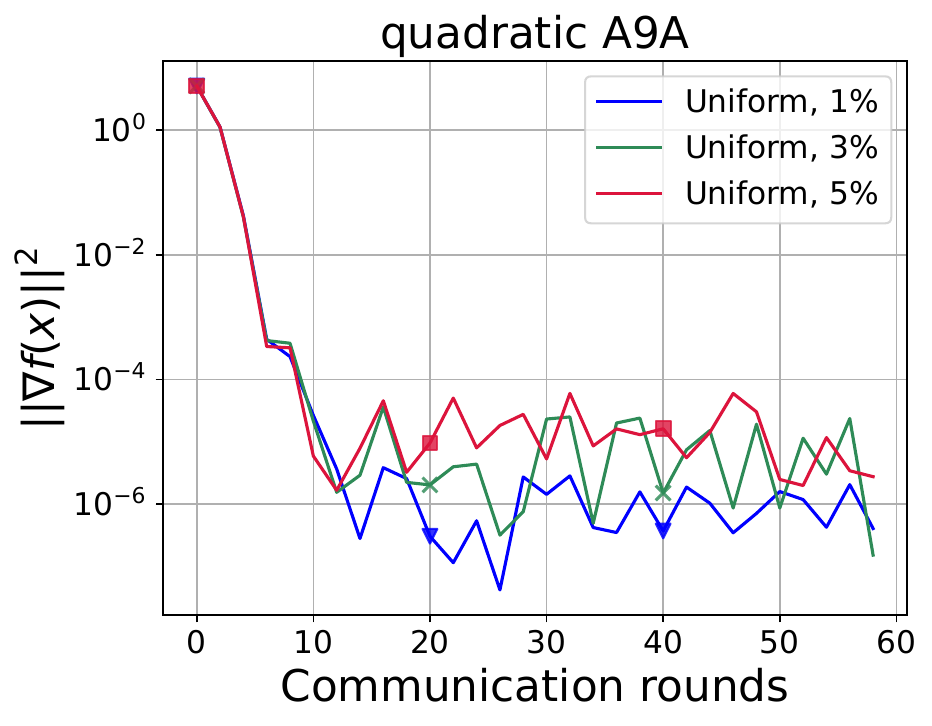}
   \end{subfigure}
   \begin{subfigure}{0.32\textwidth}
   \centering
       \includegraphics[width=\linewidth]{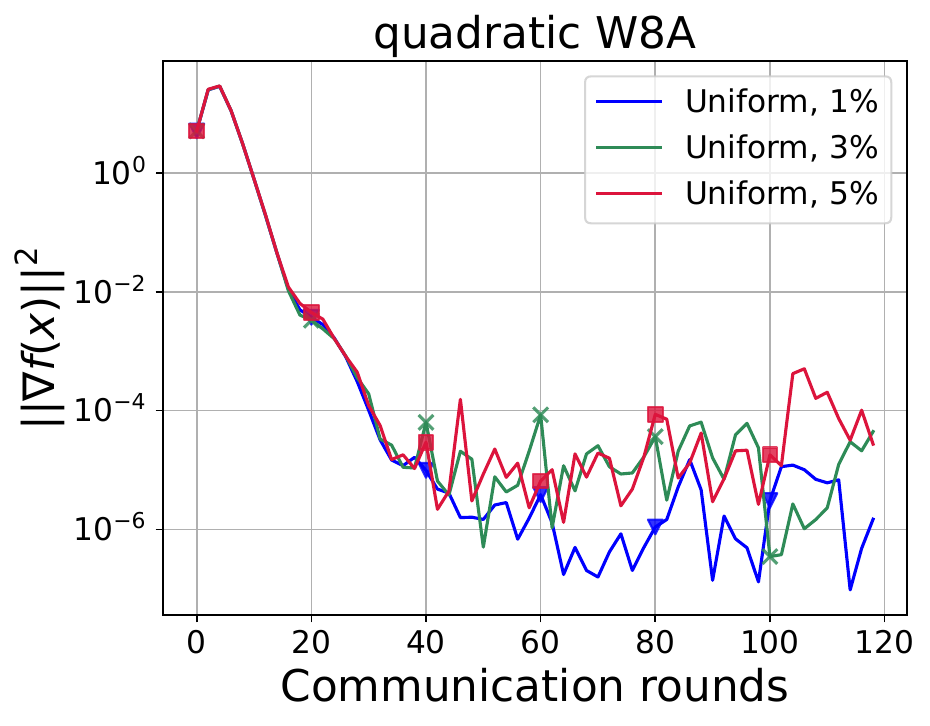}
   \end{subfigure}
   \caption{Stability of \texttt{ASEG} with \texttt{SVRG}-solver and uniform noise. The comparison is made solving quadratic problem on $M=200$ nodes}
   \label{fig:stability_quadratic_svrg}
\end{figure}

\begin{figure}[h!] % <---
   \begin{subfigure}{0.32\textwidth}
       \centering
       \includegraphics[width=\linewidth]{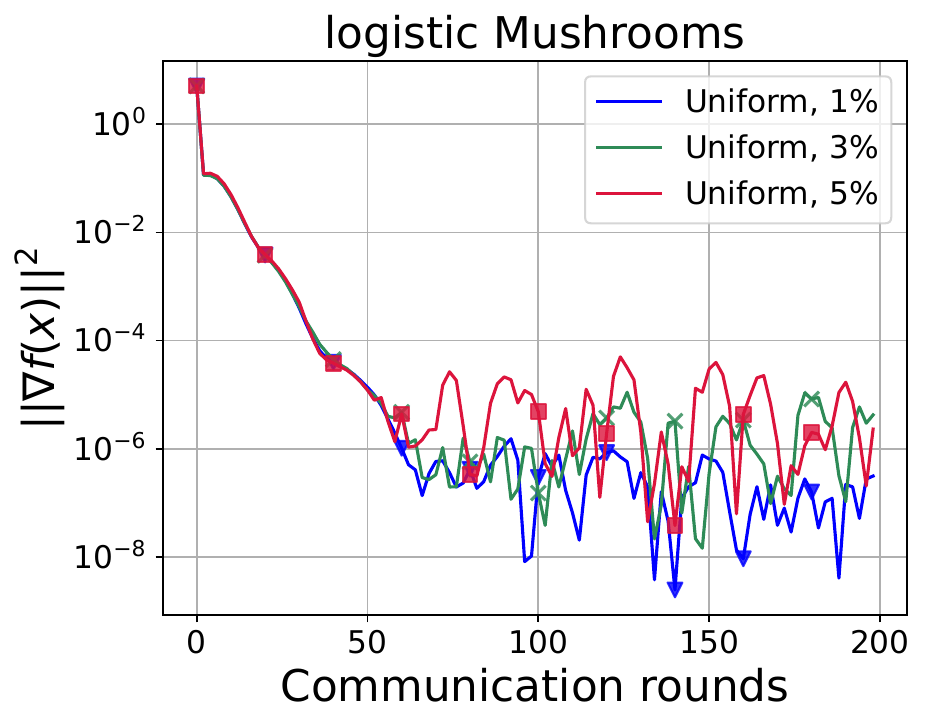}
   \end{subfigure}
   \begin{subfigure}{0.32\textwidth}
        \centering
       \includegraphics[width=\linewidth]{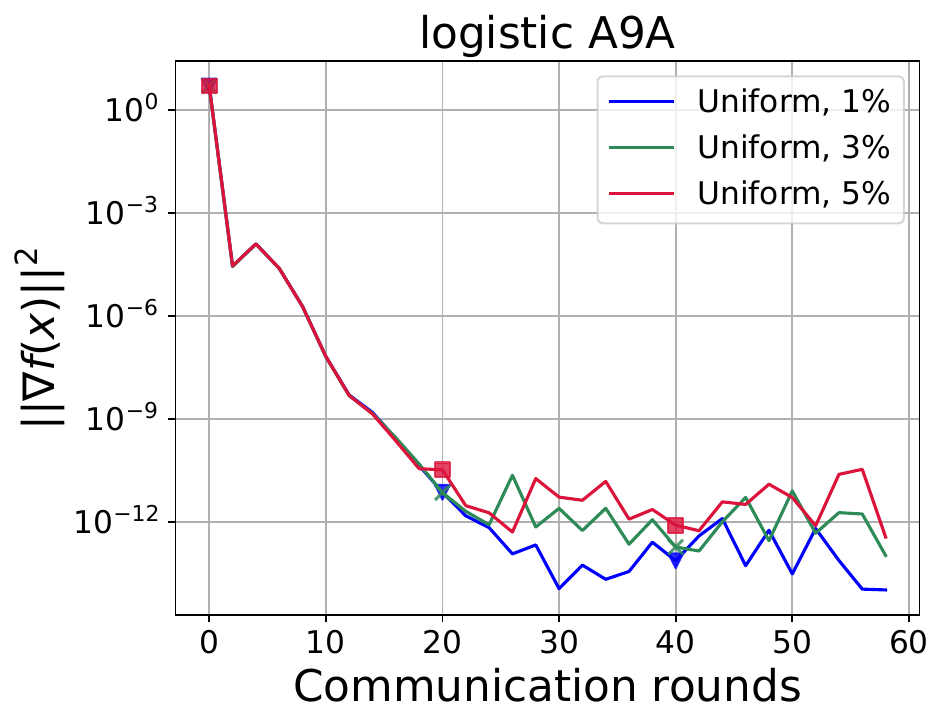}
   \end{subfigure}
   \begin{subfigure}{0.32\textwidth}
   \centering
       \includegraphics[width=\linewidth]{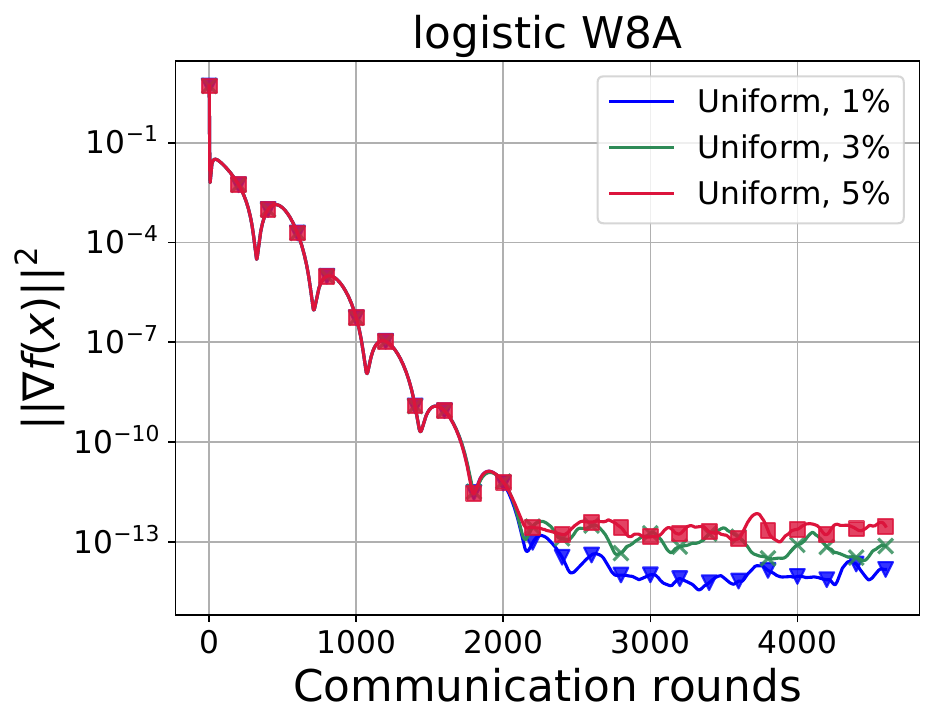}
   \end{subfigure}
   \caption{Stability of \texttt{ASEG} with \texttt{SVRG}-solver and uniform noise. The comparison is made solving logistic problem on $M=200$ nodes}
   \label{fig:stability_logistic_svrg}
\end{figure}

\begin{figure}[h!] % <---
   \begin{subfigure}{0.32\textwidth}
       \centering
       \includegraphics[width=\linewidth]{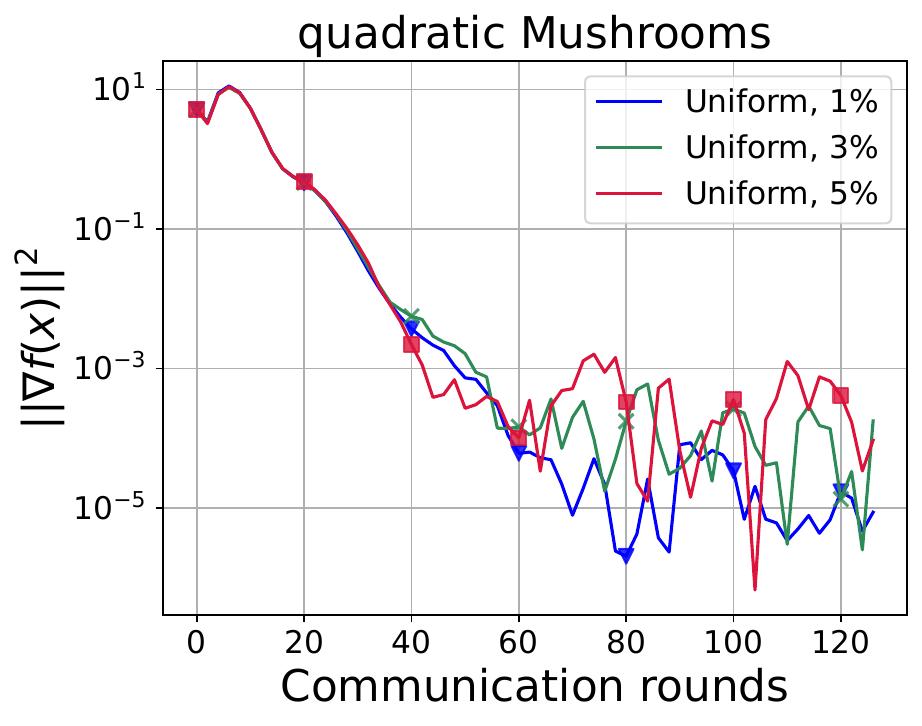}
   \end{subfigure}
   \begin{subfigure}{0.32\textwidth}
        \centering
       \includegraphics[width=\linewidth]{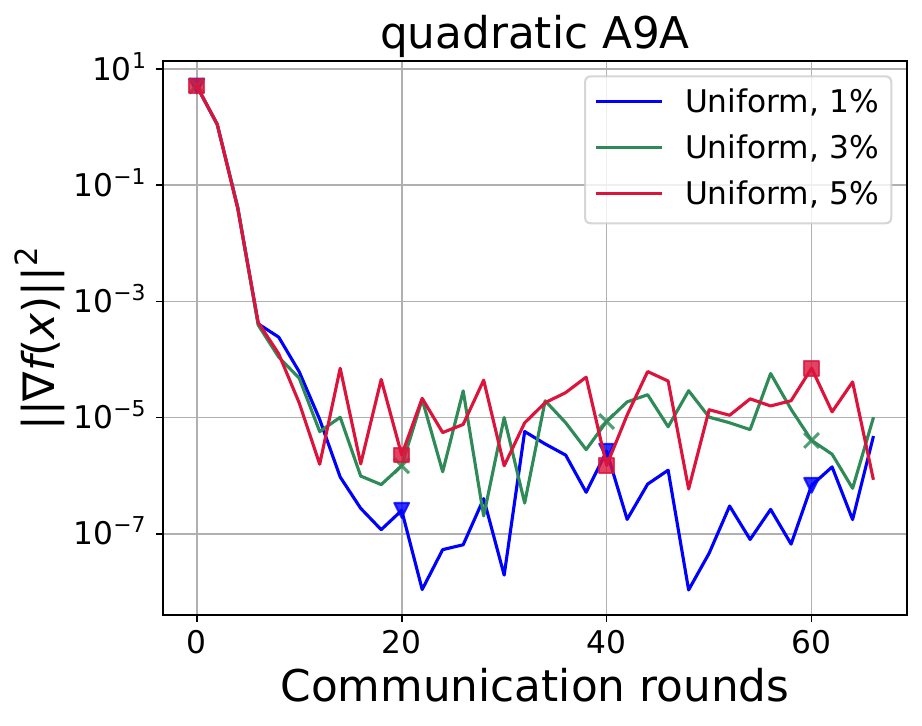}
   \end{subfigure}
   \begin{subfigure}{0.32\textwidth}
   \centering
       \includegraphics[width=\linewidth]{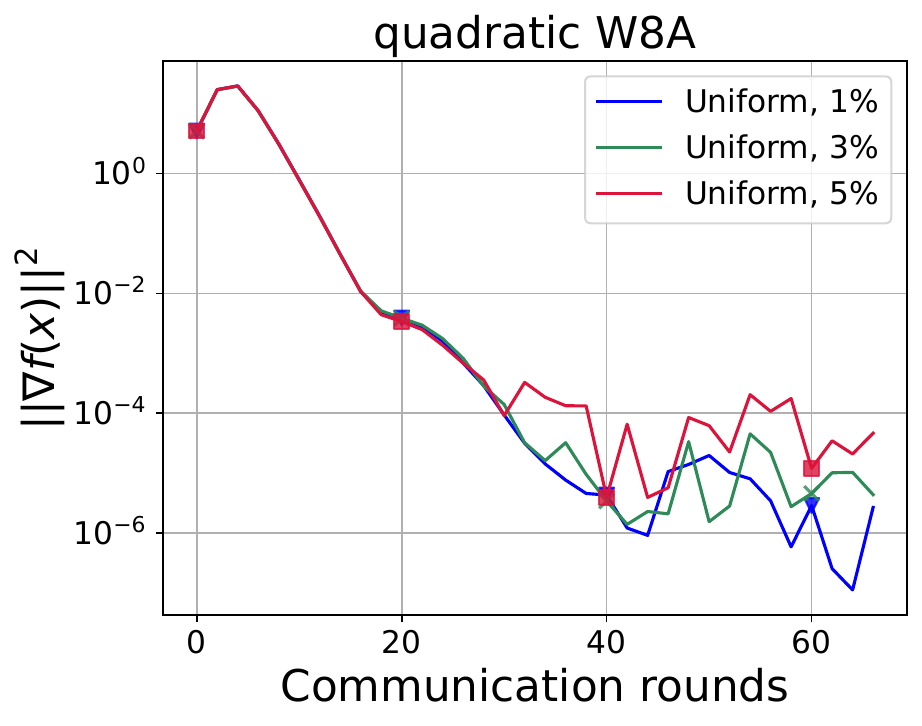}
   \end{subfigure}
   \caption{Stability of \texttt{ASEG} with \texttt{SARAH}-solver and uniform noise. The comparison is made solving quadratic problem on $M=200$ nodes}
   \label{fig:stability_quadratic_sarah}
\end{figure}

\begin{figure}[h!] % <---
   \begin{subfigure}{0.32\textwidth}
       \centering
       \includegraphics[width=\linewidth]{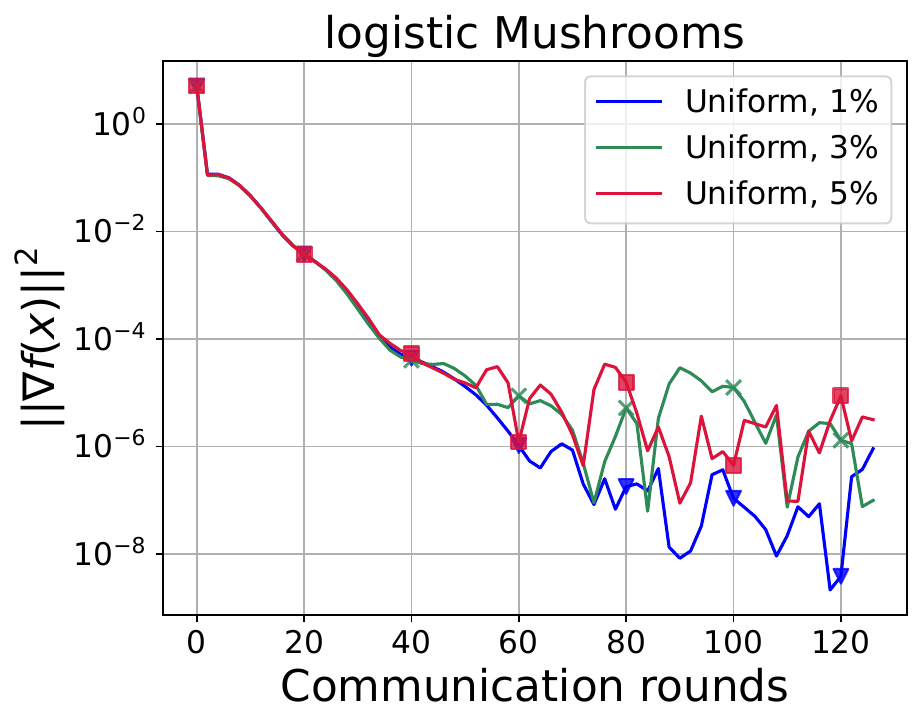}
   \end{subfigure}
   \begin{subfigure}{0.32\textwidth}
        \centering
       \includegraphics[width=\linewidth]{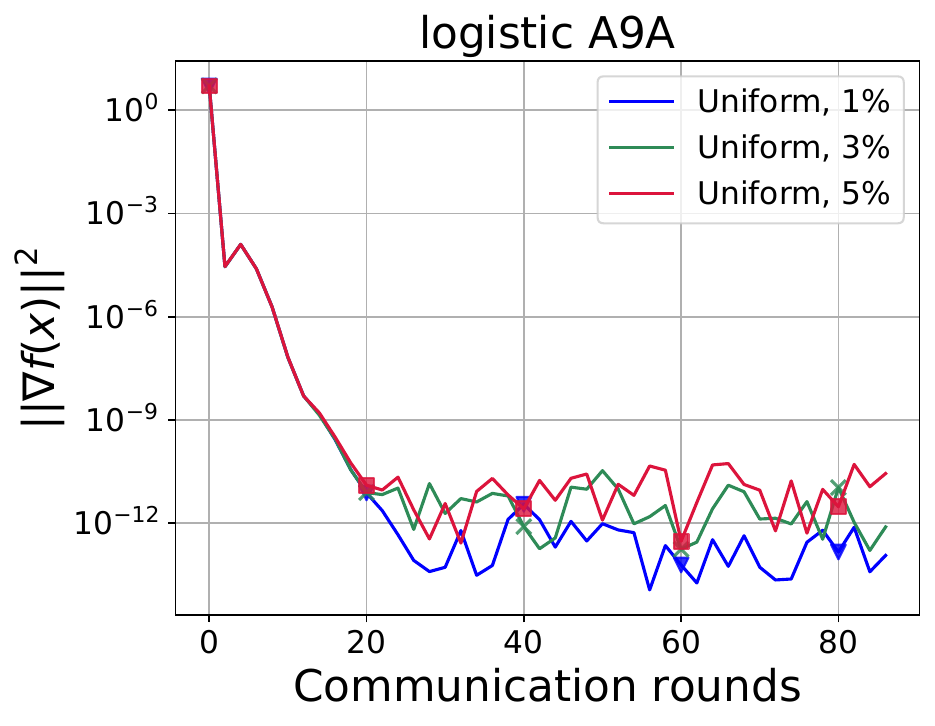}
   \end{subfigure}
   \begin{subfigure}{0.32\textwidth}
   \centering
       \includegraphics[width=\linewidth]{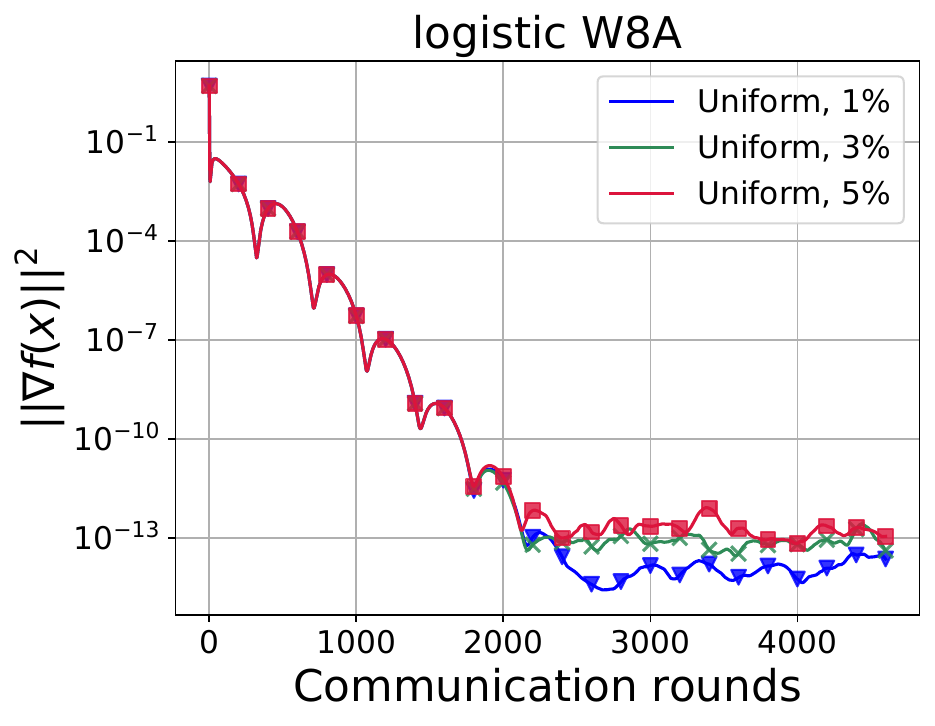}
   \end{subfigure}
   \caption{Stability of \texttt{ASEG} with \texttt{SARAH}-solver and uniform noise. The comparison is made solving logistic problem on $M=200$ nodes}
   \label{fig:stability_logistic_sarah}
\end{figure}

$\bullet$ \texttt{ASEG} experiments test how changing the epoch size changes the \texttt{ASEG} and subproblem (Line \ref{st_ae:line:2} of Algorithm \ref{st_ae:alg}) convergence with \texttt{SVRG} and \texttt{SARAH} solvers (Figure \ref{fig:quadratic_svrg_aseg}, Figure \ref{fig:logistic_svrg_aseg}, Figure \ref{fig:quadratic_svrg_subtask}, Figure \ref{fig:logistic_svrg_subtask}, Figure \ref{fig:quadratic_sarah_aseg}, Figure \ref{fig:logistic_sarah_aseg}, Figure \ref{fig:quadratic_sarah_subtask}, Figure \ref{fig:logistic_sarah_subtask}).

\begin{figure}[h!] % <---
   \begin{subfigure}{0.32\textwidth}
       \centering
       \includegraphics[width=\linewidth]{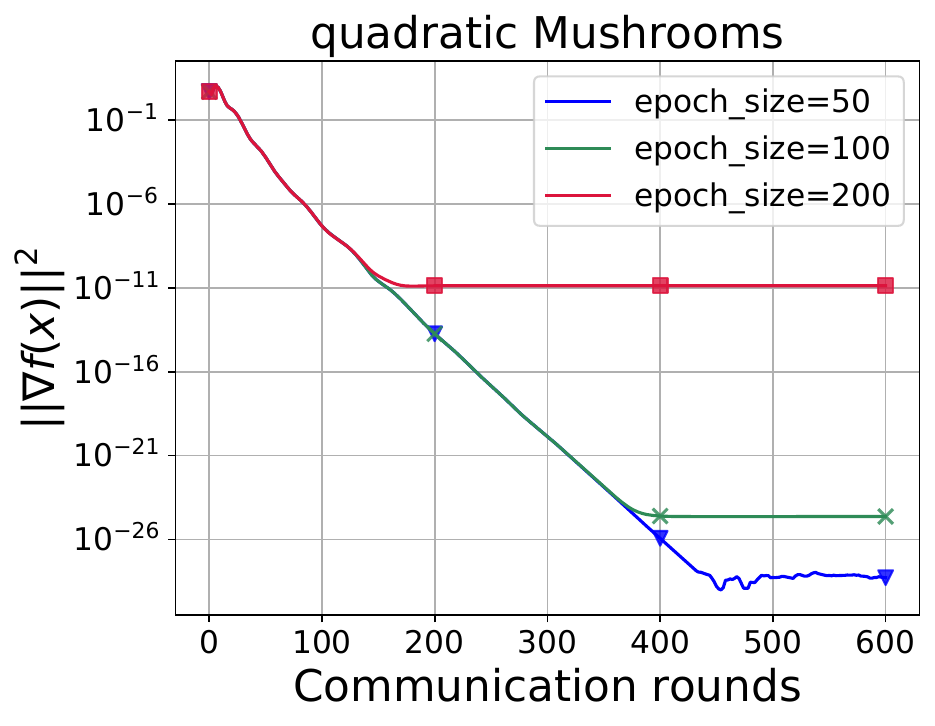}
   \end{subfigure}
   \begin{subfigure}{0.32\textwidth}
        \centering
       \includegraphics[width=\linewidth]{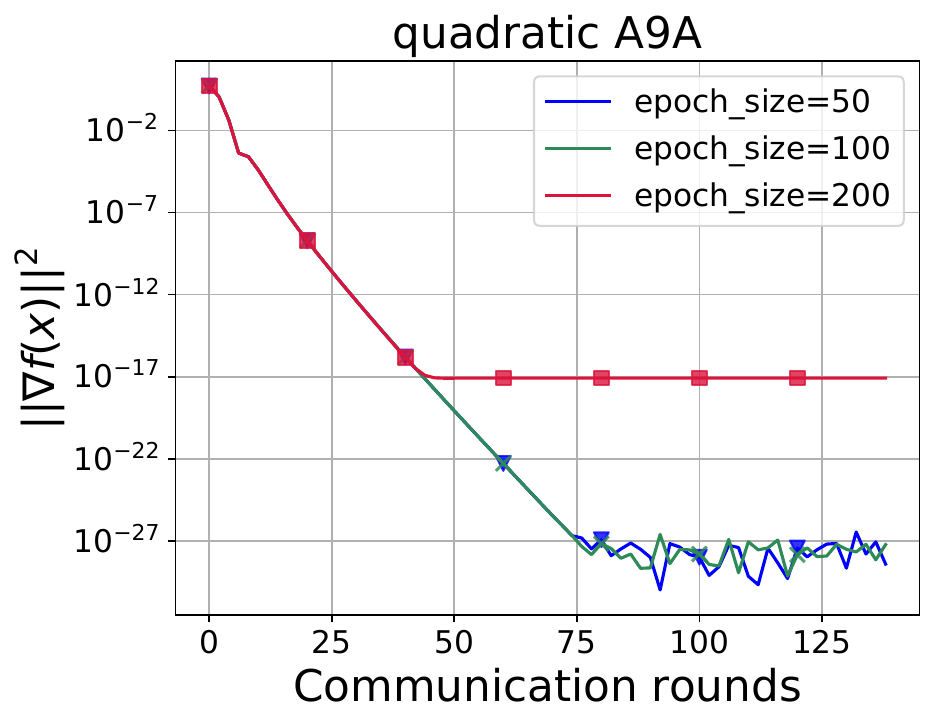}
   \end{subfigure}
   \begin{subfigure}{0.32\textwidth}
   \centering
       \includegraphics[width=\linewidth]{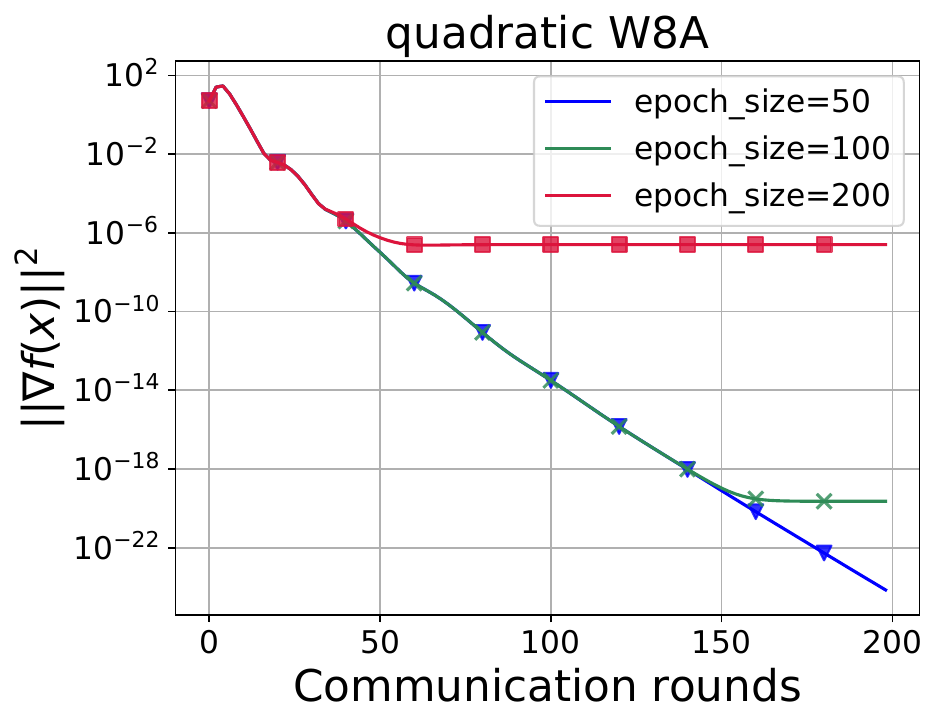}
   \end{subfigure}
   \caption{\texttt{ASEG} with \texttt{SVRG}-solver. The comparison is made solving quadratic problem on $M=200$ nodes with different epoch sizes of solver}
   \label{fig:quadratic_svrg_aseg}
\end{figure}

\begin{figure}[h!] % <---
   \begin{subfigure}{0.32\textwidth}
       \centering
       \includegraphics[width=\linewidth]{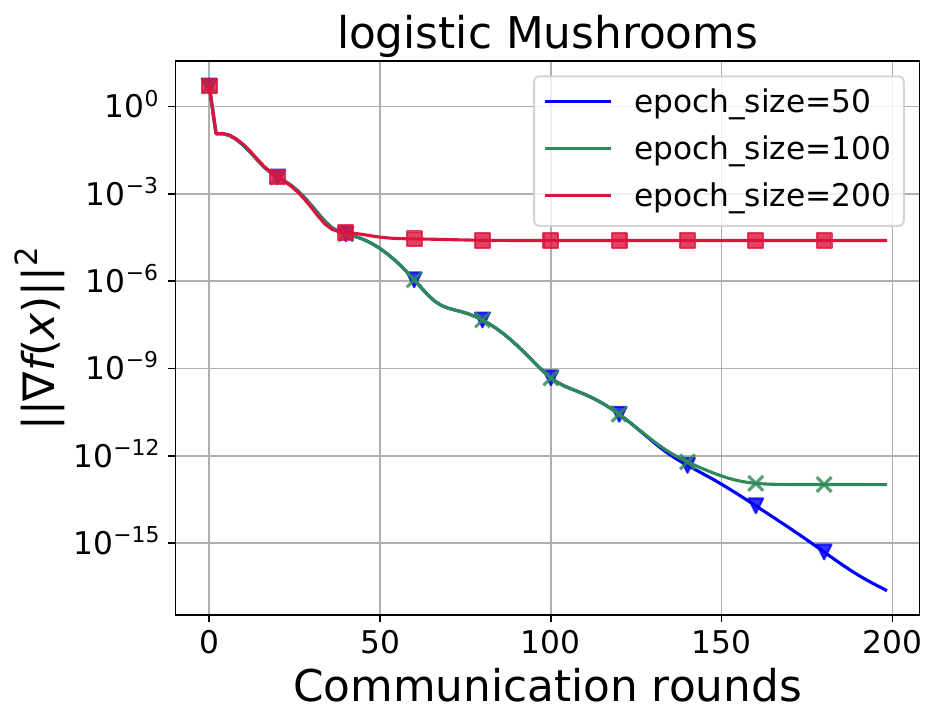}
   \end{subfigure}
   \begin{subfigure}{0.32\textwidth}
        \centering
       \includegraphics[width=\linewidth]{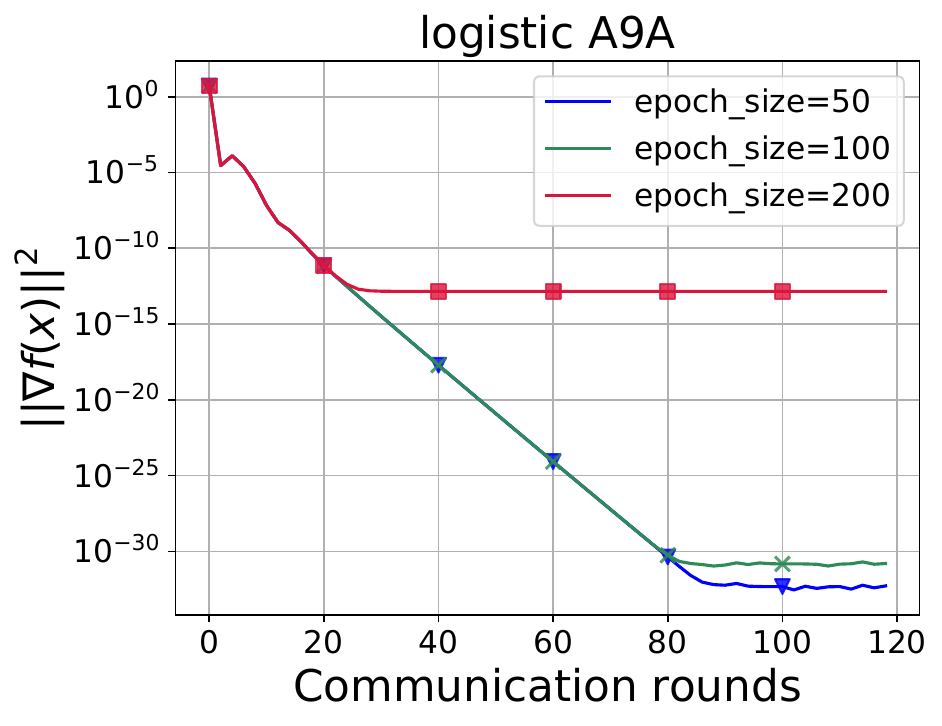}
   \end{subfigure}
   \begin{subfigure}{0.32\textwidth}
   \centering
       \includegraphics[width=\linewidth]{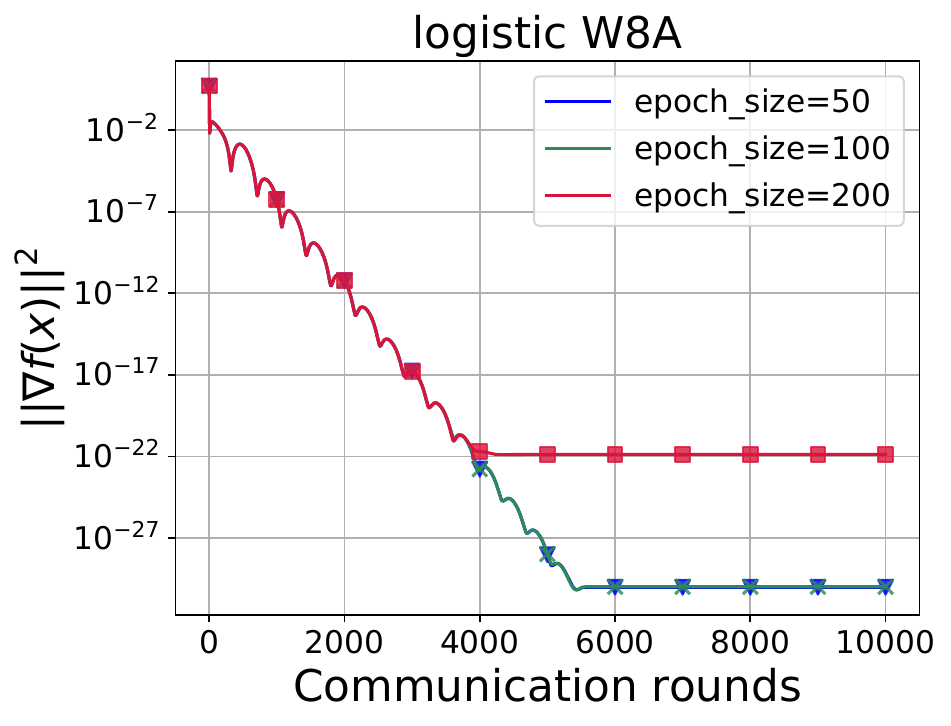}
   \end{subfigure}
   \caption{\texttt{ASEG} with \texttt{SVRG}-solver. The comparison is made solving logistic problem on $M=200$ nodes with different epoch sizes of solver}
   \label{fig:logistic_svrg_aseg}
\end{figure}

\begin{figure}[h!] % <---
   \begin{subfigure}{0.32\textwidth}
       \centering
       \includegraphics[width=\linewidth]{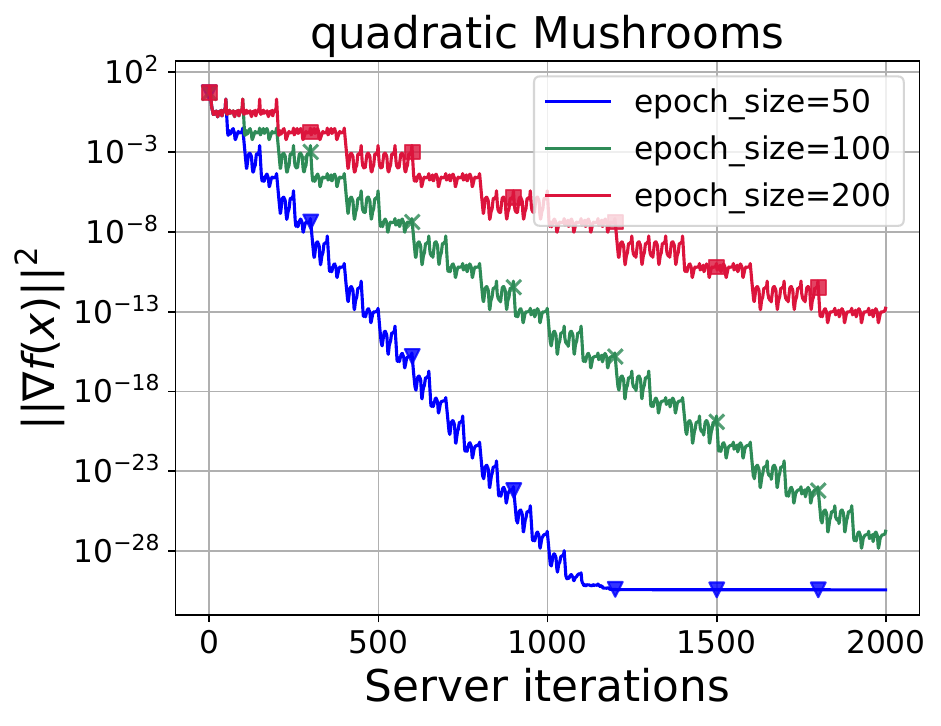}
   \end{subfigure}
   \begin{subfigure}{0.32\textwidth}
        \centering
       \includegraphics[width=\linewidth]{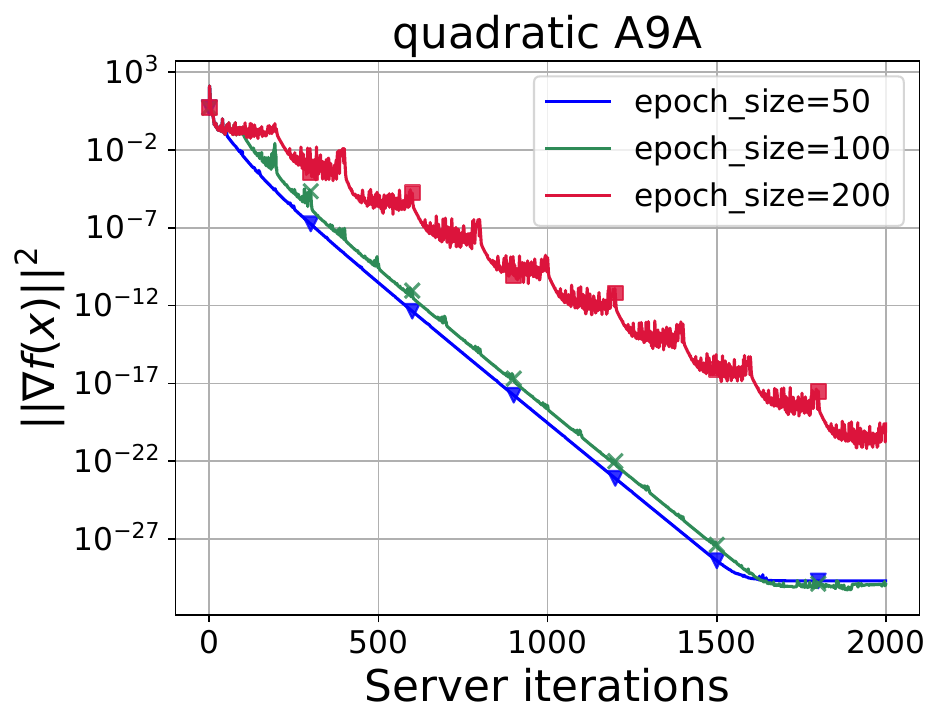}
   \end{subfigure}
   \begin{subfigure}{0.32\textwidth}
   \centering
       \includegraphics[width=\linewidth]{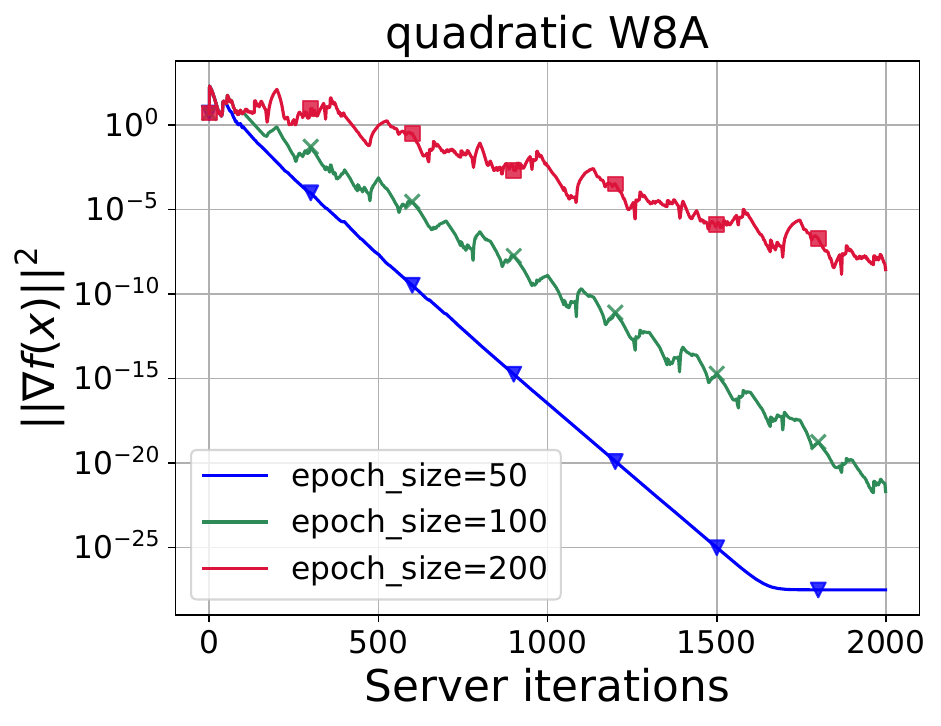}
   \end{subfigure}
   \caption{Subtask with \texttt{SVRG}-solver. The comparison is made solving quadratic problem on $M=200$ nodes with different epoch sizes of solver}
   \label{fig:quadratic_svrg_subtask}
\end{figure}
\begin{figure}[h!] % <---
   \begin{subfigure}{0.32\textwidth}
       \centering
       \includegraphics[width=\linewidth]{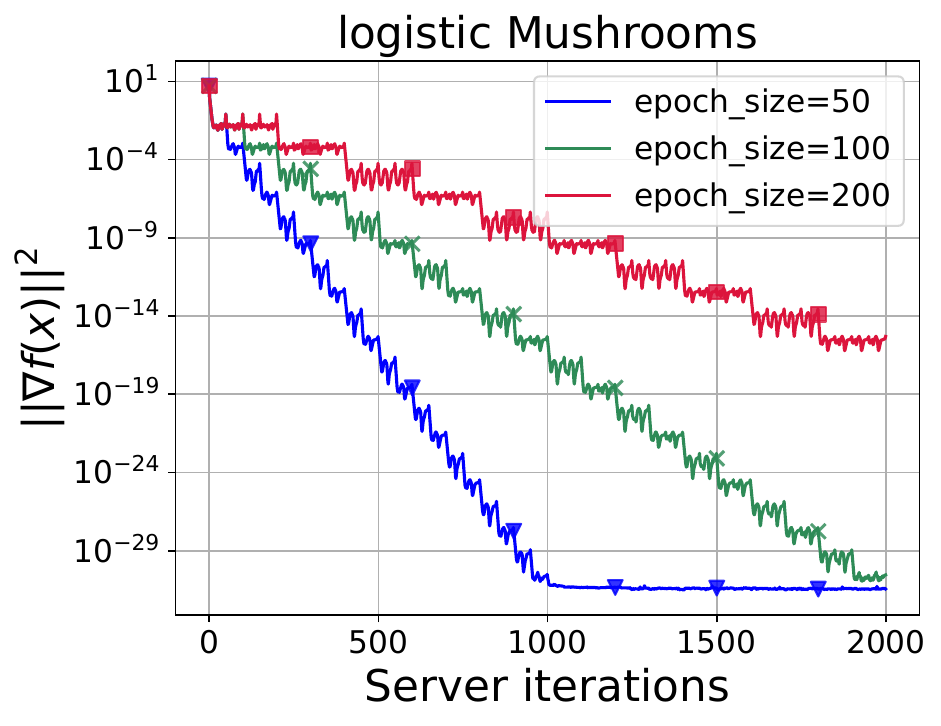}
   \end{subfigure}
   \begin{subfigure}{0.32\textwidth}
        \centering
       \includegraphics[width=\linewidth]{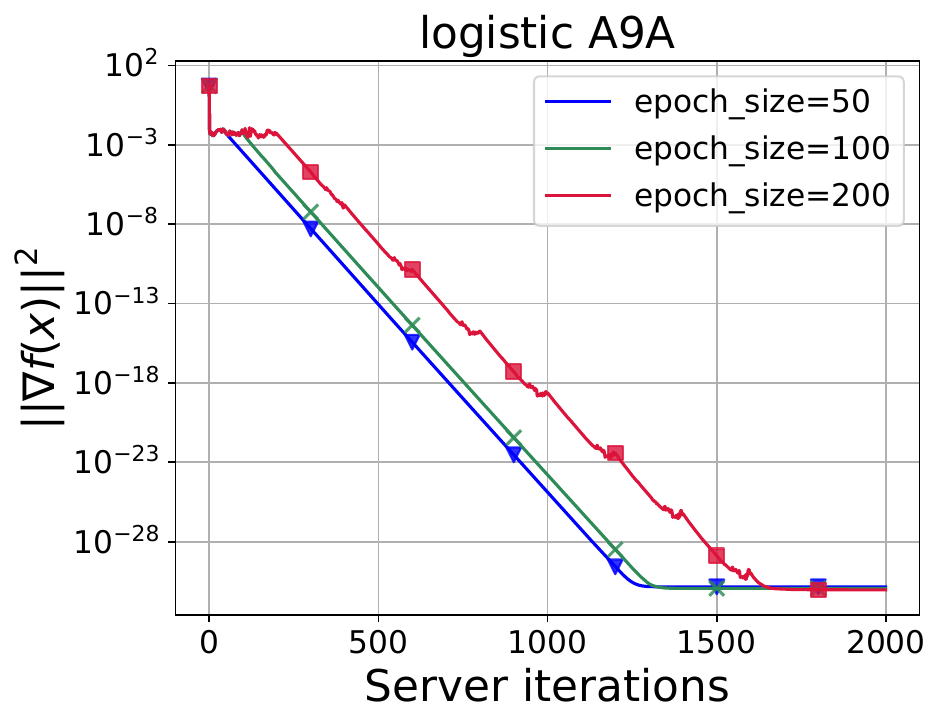}
   \end{subfigure}
   \begin{subfigure}{0.32\textwidth}
   \centering
       \includegraphics[width=\linewidth]{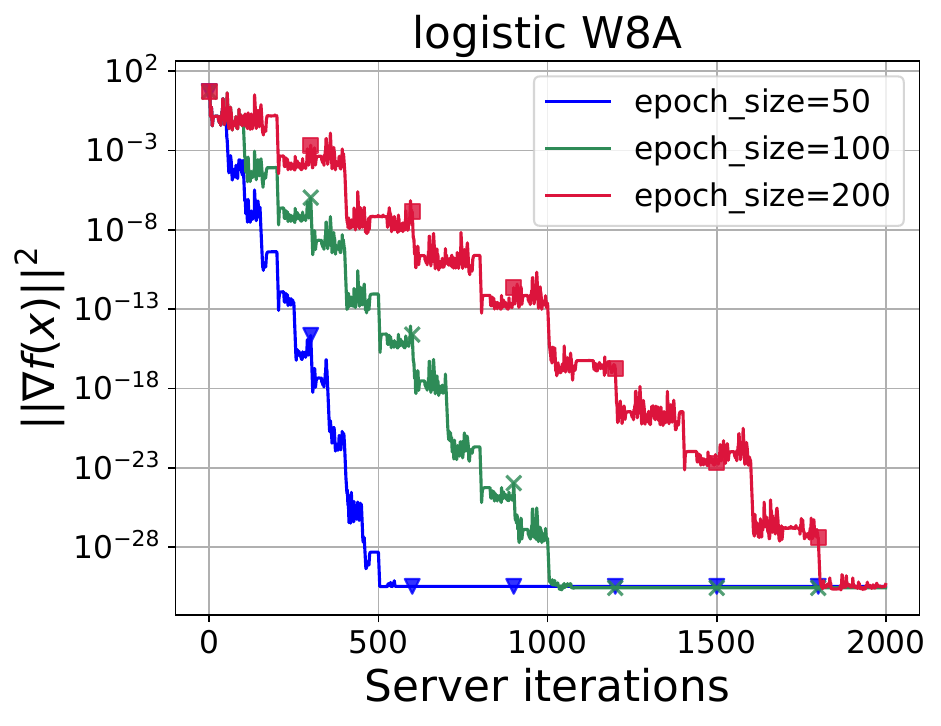}
   \end{subfigure}
   \caption{Subtask with \texttt{SVRG}-solver. The comparison is made solving logistic problem on $M=200$ nodes with different epoch sizes of solver}
   \label{fig:logistic_svrg_subtask}
\end{figure}

\begin{figure}[h!] % <---
   \begin{subfigure}{0.32\textwidth}
       \centering
       \includegraphics[width=\linewidth]{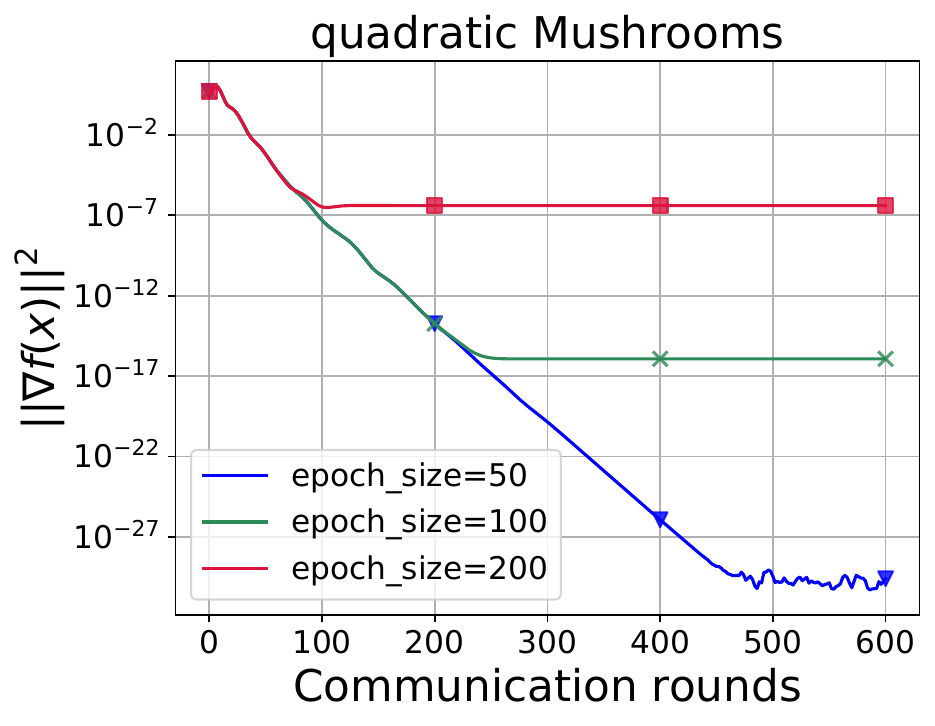}
   \end{subfigure}
   \begin{subfigure}{0.32\textwidth}
        \centering
       \includegraphics[width=\linewidth]{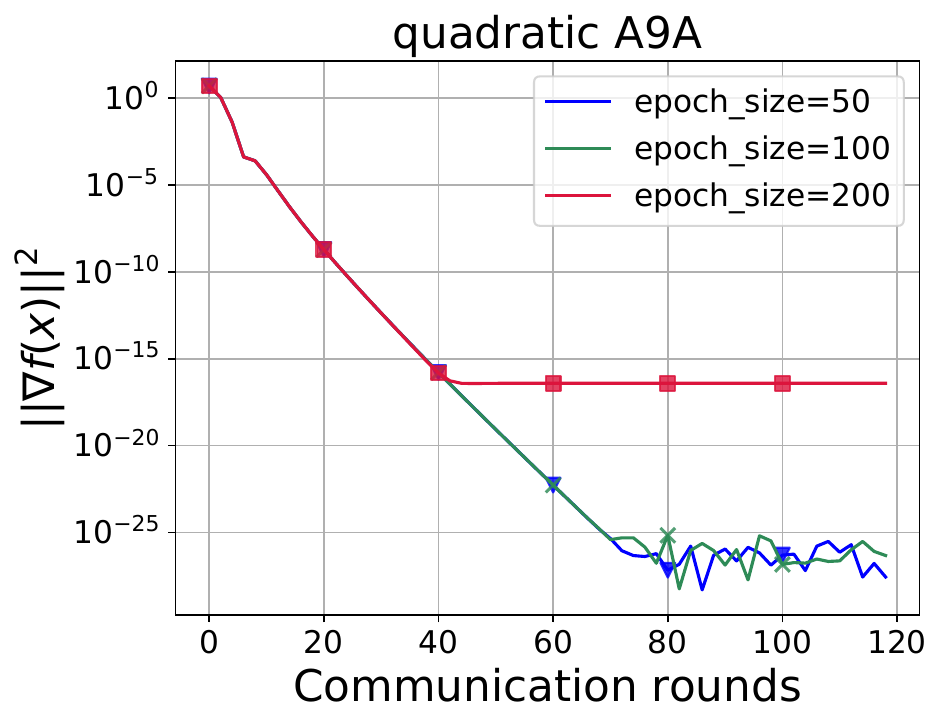}
   \end{subfigure}
   \begin{subfigure}{0.32\textwidth}
   \centering
       \includegraphics[width=\linewidth]{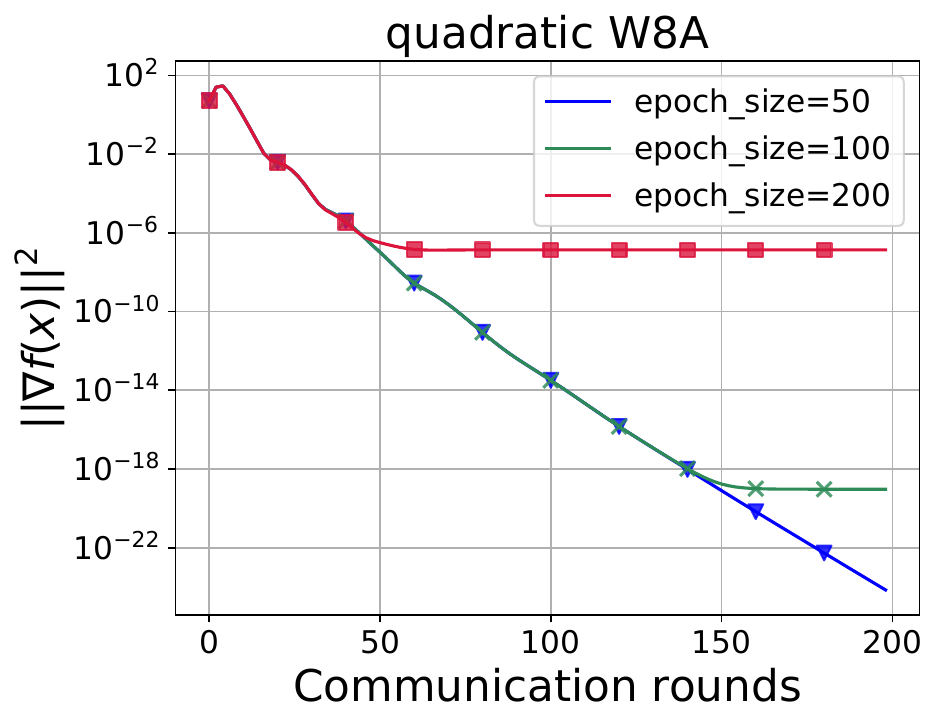}
   \end{subfigure}
   \caption{\texttt{ASEG} with \texttt{SARAH}-solver. The comparison is made solving quadratic problem on $M=200$ nodes with different epoch sizes of solver}
   \label{fig:quadratic_sarah_aseg}
\end{figure}
\begin{figure}[h!] % <---
   \begin{subfigure}{0.32\textwidth}
       \centering
       \includegraphics[width=\linewidth]{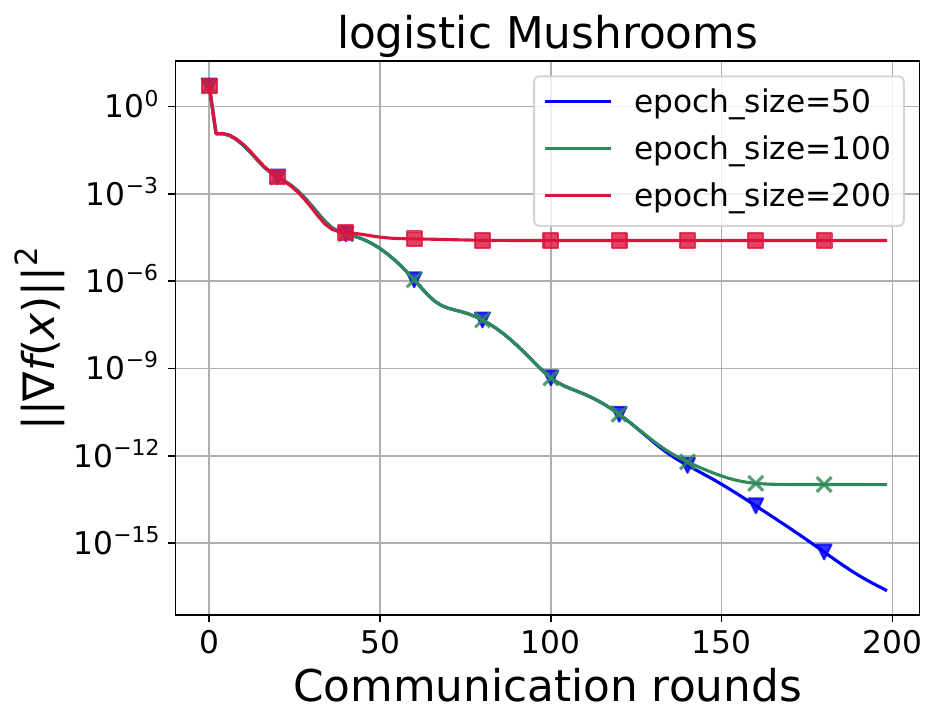}
   \end{subfigure}
   \begin{subfigure}{0.32\textwidth}
        \centering
       \includegraphics[width=\linewidth]{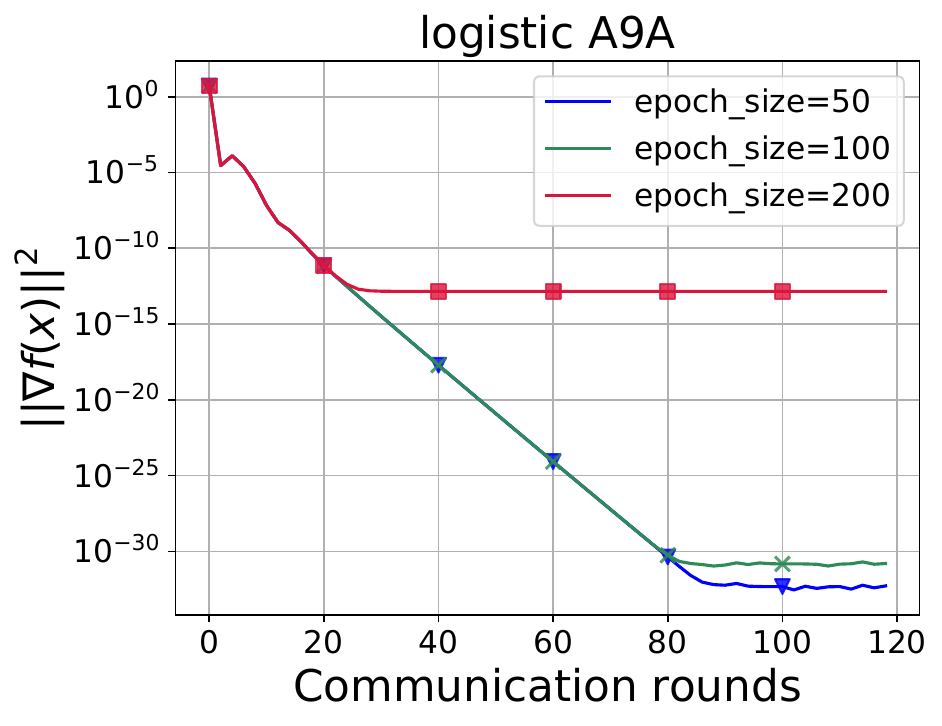}
   \end{subfigure}
   \begin{subfigure}{0.32\textwidth}
   \centering
       \includegraphics[width=\linewidth]{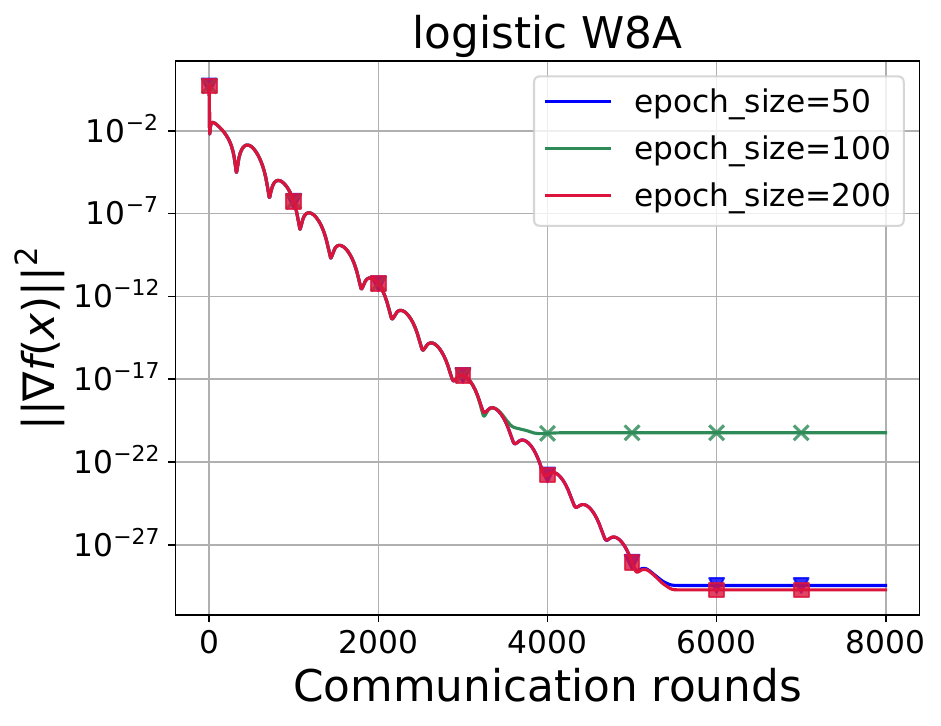}
   \end{subfigure}
   \caption{\texttt{ASEG} with \texttt{SARAH}-solver. The comparison is made solving logistic problem on $M=200$ nodes with different epoch sizes of solver}
   \label{fig:logistic_sarah_aseg}
\end{figure}

\begin{figure}[h!] % <---
   \begin{subfigure}{0.32\textwidth}
       \centering
       \includegraphics[width=\linewidth]{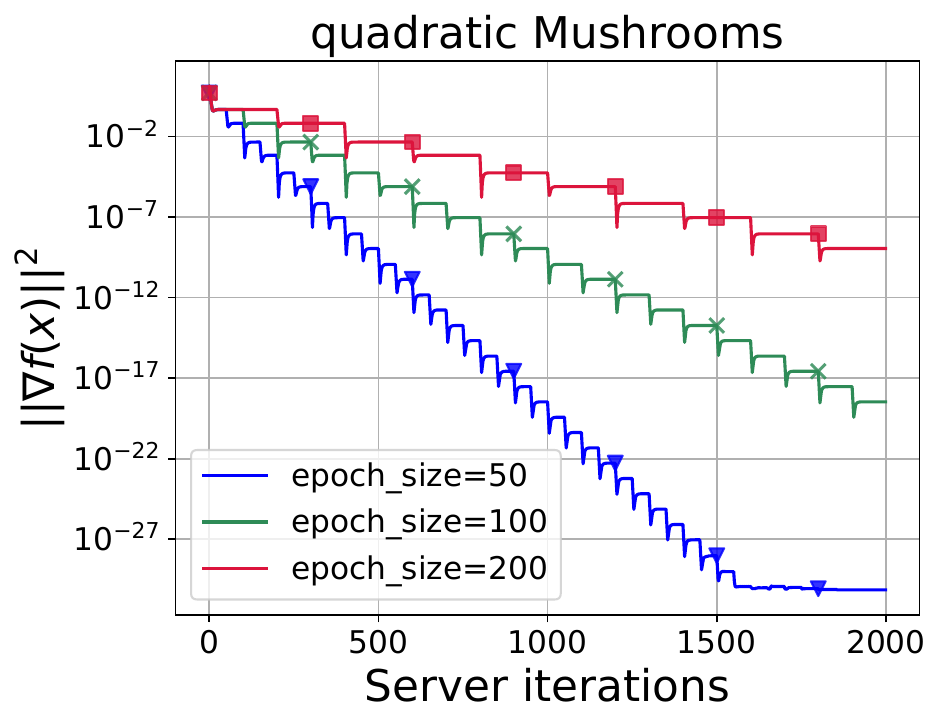}
   \end{subfigure}
   \begin{subfigure}{0.32\textwidth}
        \centering
       \includegraphics[width=\linewidth]{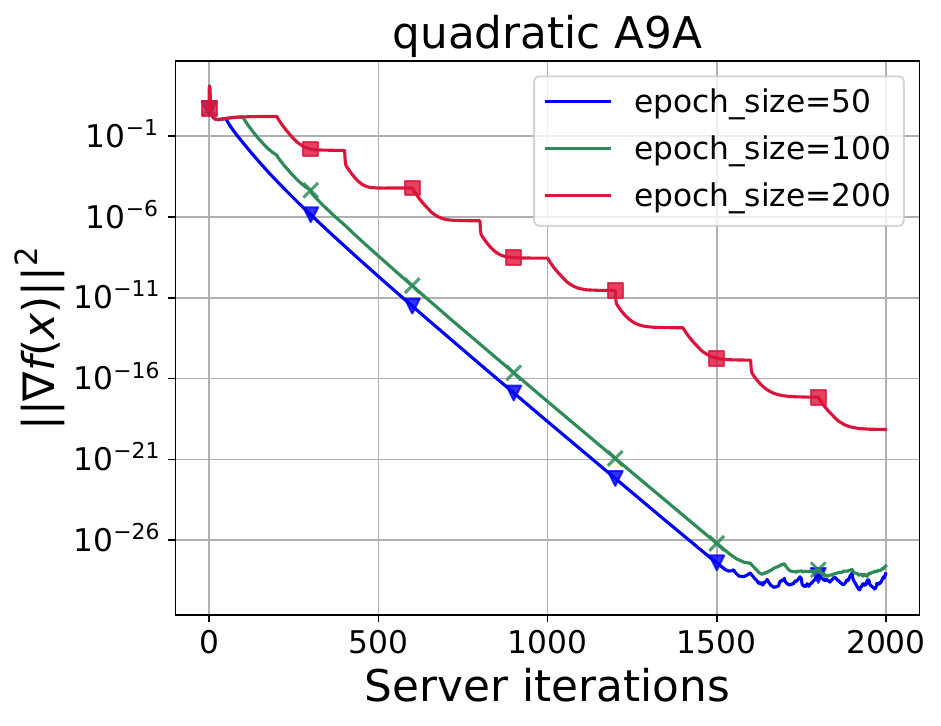}
   \end{subfigure}
   \begin{subfigure}{0.32\textwidth}
   \centering
       \includegraphics[width=\linewidth]{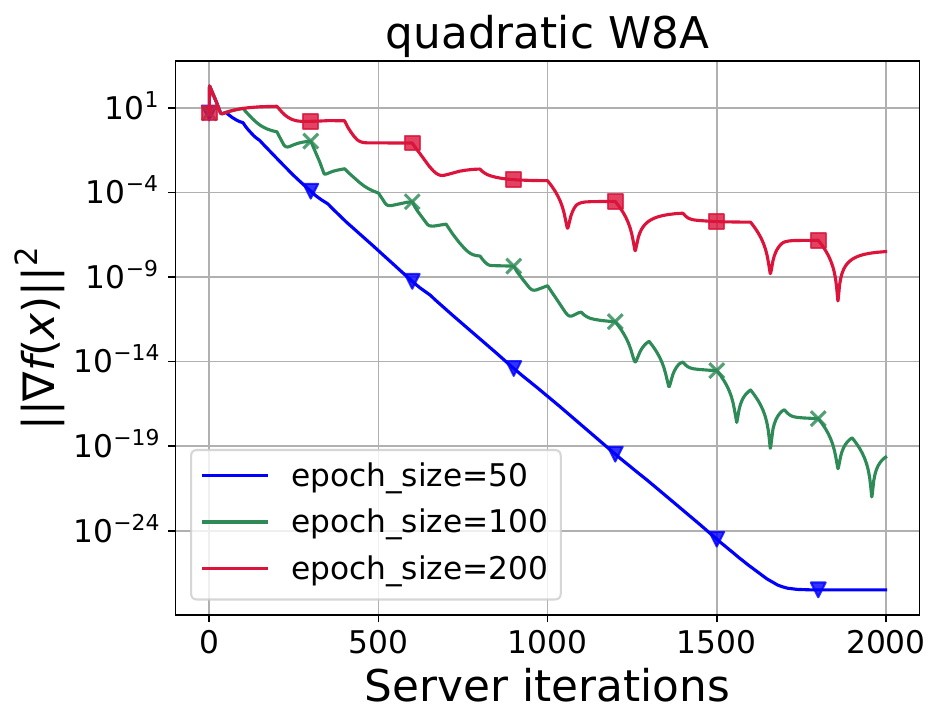}
   \end{subfigure}
   \caption{Subtask with \texttt{SARAH}-solver. The comparison is made solving quadratic problem on $M=200$ nodes with different epoch sizes of solver}
   \label{fig:quadratic_sarah_subtask}
\end{figure}
\begin{figure}[h!] % <---
   \begin{subfigure}{0.32\textwidth}
       \centering
       \includegraphics[width=\linewidth]{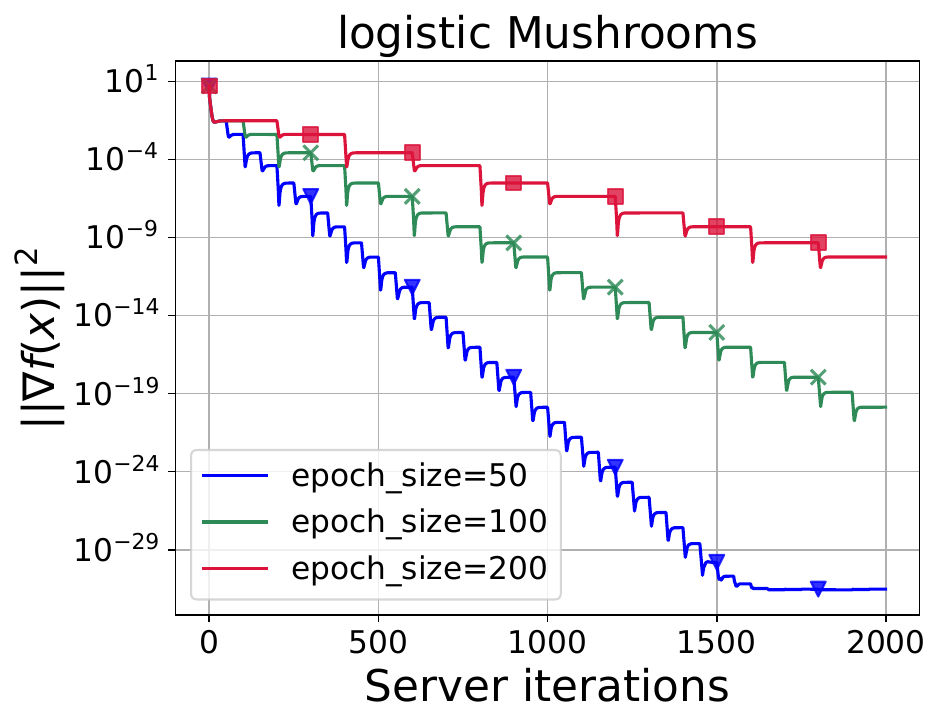}
   \end{subfigure}
   \begin{subfigure}{0.32\textwidth}
        \centering
       \includegraphics[width=\linewidth]{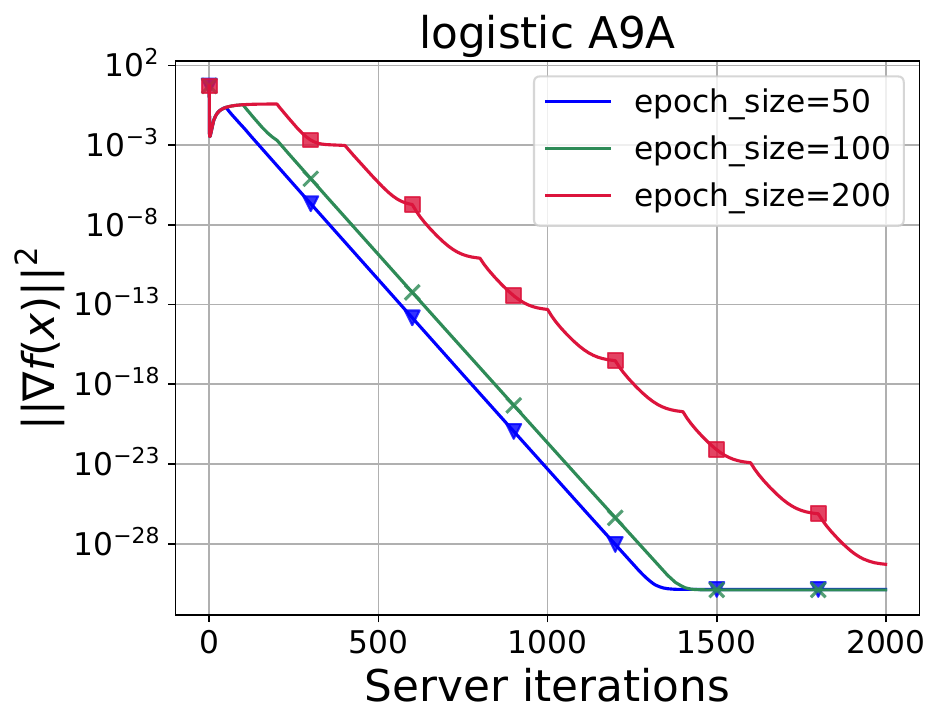}
   \end{subfigure}
   \begin{subfigure}{0.32\textwidth}
   \centering
       \includegraphics[width=\linewidth]{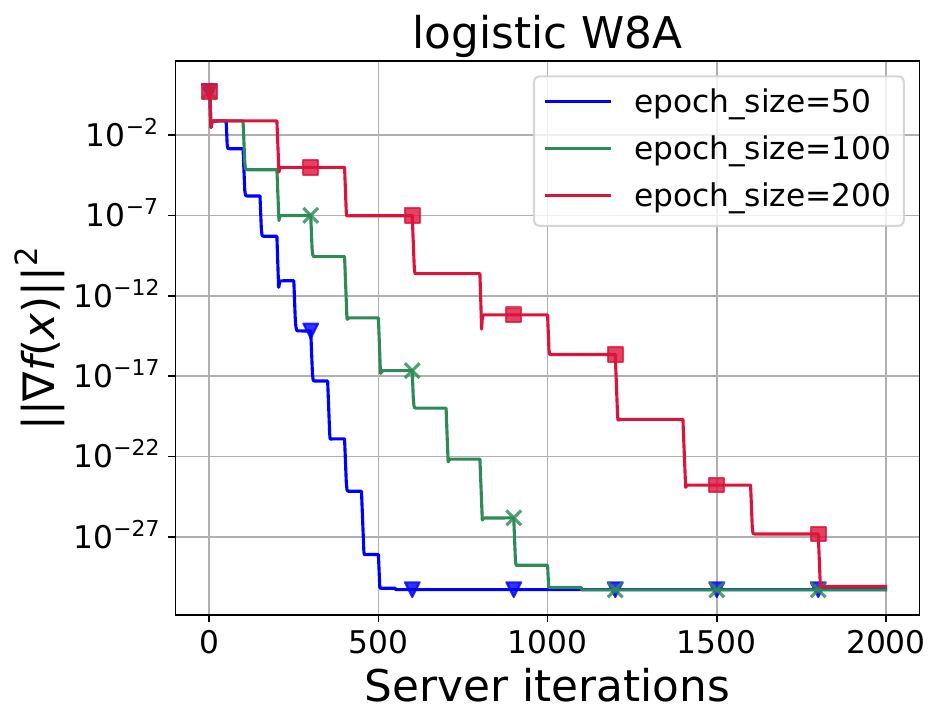}
   \end{subfigure}
   \caption{Subtask with \texttt{SARAH}-solver. The comparison is made solving logistic problem on $M=200$ nodes with different epoch sizes of solver}
   \label{fig:logistic_sarah_subtask}
\end{figure}

\section{Conclusion}
Based on the \texttt{SGD}-like approach, we have adapted \texttt{Accelerated ExtraGradient} to solve Federated Learning problems by adding two stochasticities: sampling of computational nodes and application of noise to local gradients. Numerical experiments show that for small values of $\nicefrac{\delta}{\mu}$ or a sufficiently large number of available nodes, Algorithm \ref{st_ae:alg} (\texttt{ASEG}) is significantly superior to accelerated variance reduction methods (Figure \ref{aseg_vs_svrs_a9a}), which is consistent with the theory (Corollary \ref{st_ae:cor_1}). If the problem parameters are not "good" enough, there is a $\nicefrac{\delta^2}{(\mu^2M)}\cdot\log\nicefrac{1}{\varepsilon}$ term in the communication complexity of batched \texttt{ASEG} (Corollary \ref{st_ae:cor_2}). Nevertheless, the superiority in the number of communications remains even in this case (Figure \ref{aseg_vs_svrs_mush}, Figure \ref{aseg_vs_svrs_a9a}). Experiments also show the robustness of our method against noise (Figure \ref{fig:stability_quadratic_svrg}, Figure \ref{fig:stability_logistic_svrg}, Figure \ref{fig:stability_quadratic_sarah}, Figure \ref{fig:stability_logistic_sarah}). The server subproblem has also been studied numerically (Figure \ref{fig:quadratic_svrg_subtask}, Figure \ref{fig:logistic_svrg_subtask}, Figure  \ref{fig:quadratic_sarah_subtask}, Figure \ref{fig:logistic_sarah_subtask}) and theoretically. The analysis shows the inefficiency of randomly choosing the number of solver iterations for the internal subtask (Proposition \ref{prop:1}).

\subsection*{Acknowledgements}

The work was done in the Laboratory of Federated Learning Problems of the ISP RAS (Supported by Grant App. No. 2 to Agreement No. 075-03-2024-214).

\clearpage
\appendix

% \part*{Appendix}

In all the appendices, we consider $r_1(x)=q(x)$ and $p(x)=r(x)-r_1(x)$. Mathematical expectations are implied in the context of the proof.
\section{\textbf{Lemma 1}}\label{ap:proof_th_1}

Here we prove Lemma \ref{st_ae:main_lemma}. First, we need the following:
\begin{lemma}\label{st_ae:lemma}
    Consider Algorithm \ref{st_ae:alg}. Let  $\theta$ be defined as in Lemma \ref{st_ae:main_lemma}: $\theta \leq \frac{1}{3\delta}$. Then under Assumptions \ref{ass:1} and \ref{ass:2}, the following inequality holds for all $\bar{x} \in \mathbb{R}^d$:
	\begin{equation}
	\label{st_ae:eq:1}
	\begin{split}
		\mathbb{E}\big[2\langle\bar{x} - x_g^k, t_k\rangle \big|& x^k, x_f^{k+1}\big]
		\\\leq&
		2\left[ r(\bar{x}) - r(x_f^{k+1}) \right] - \mu \lVert x_f^{k+1} - \bar{x} \rVert^2
	- \theta \lVert \nabla r(x_f^{k+1}) \rVert^2
	\\&
	+\frac{11\theta}{3} \left( \lVert \nabla \bar{A}_\theta^k(x_f^{k+1}) \rVert^2 - \frac{9\delta^2}{11} \lVert x_g^k - \arg\min_{x \in \mathbb{R}^d} \bar{A}_\theta^k(x) \rVert^2 \right) 
    \\
    &+ 3 \theta \lVert \nabla p(x_g^k) - s_k \rVert ^ 2.
	\end{split}
	\end{equation}
\end{lemma}

\begin{proof}
Using the $\mu$-strong convexity of $r(x)$, we get

\begin{align*}
\mathbb{E}\big[2\langle\bar{x} - x_g^k, t_k\rangle \big| x^k, x_f^{k+1}\big] =& 2\langle\bar{x} - x_g^k, \mathbb{E}\big[t_k \big| x^k, x_f^{k+1}\big]\rangle
 \\ =& 2\langle\bar{x} - x_g^k, \nabla r(x_f^{k+1})\rangle \\=& 2\langle\bar{x} - x_f^{k+1}, \nabla r(x_f^{k+1})\rangle
+2\langle x_f^{k+1} - x_g^k, \nabla r(x_f^{k+1})\rangle \\
\leq& 2\left[r(\bar{x}) - r(x_f^{k+1})\right] - \mu \lVert x_f^{k+1} - \bar{x} \rVert^2
\\&+2\langle x_f^{k+1} - x_g^k, \nabla r(x_f^{k+1})\rangle \\
=& 2\left[r(\bar{x}) - r(x_f^{k+1})\right] - \mu \lVert x_f^{k+1} - \bar{x} \rVert^2
\\&+2\theta\left\langle\theta^{-1}(x_f^{k+1} - x_g^k), \nabla r(x_f^{k+1})\right\rangle.
\end{align*}
From the properties of the norm we can write
\begin{align*}
    \left\|\frac{\x-x_g^k}{\theta} + \nabla r(\x)\right\|^2=&\frac{1}{\theta^2}\|\x-x_g^k\|^2 + \|\nabla r(\x)\|^2 \\&+ \frac{2}{\theta}\left\langle \x-x_g^k,\nabla r(\x) \right\rangle.
\end{align*}
Thus,
\begin{align*}
    \mathbb{E}\big[2\langle\bar{x} - x_g^k, t_k\rangle \big| x^k, x_f^{k+1}\big] =& 2\left[r(\bar{x}) - r(x_f^{k+1})\right] - \mu \lVert x_f^{k+1} - \bar{x} \rVert^2 -\frac{1}{\theta}\lVert x_f^{k+1} - x_g^k \rVert^2 \\&- \theta \lVert \nabla r(x_f^{k+1}) \rVert^2 + \theta \lVert \theta^{-1}(x_f^{k+1} - x_g^k) + \nabla r(x_f^{k+1}) \rVert^2.
\end{align*}
The definition of $\bar{A}_\theta^k(x)$ and $\delta$-Lipschitzness of $\nabla p$ give
\begin{align*}
\mathbb{E}\big[2\langle\bar{x} - x_g^k, t_k\rangle \big|& x^k, x_f^{k+1}\big]
\\\leq& 2\left[r(\bar{x}) - r(x_f^{k+1})\right] - \mu \lVert x_f^{k+1} - \bar{x}\rVert^2
-\frac{1}{\theta}\lVert x_f^{k+1} - x_g^k\rVert^2
\\&+\theta\lVert \nabla \bar{A}_\theta^k(x_f^{k+1}) + \nabla p(x_f^{k+1}) - \nabla p(x_g^{k}) + \nabla p(x_g^{k}) - s_k)\rVert^2 \\&- \theta\lVert \nabla r(x_f^{k+1})\rVert^2
\end{align*}
Note that
\begin{align*}
    \|a+b+c\|^2\leq3\|a\|^2+3\|b\|^2+3\|c\|^2,
\end{align*}
and the Hessian similarity (Definition \ref{def:sim}) implies 
\begin{align*}
    \|\nabla p(\x)-\nabla p(x_g^k)\|^2 \leq \delta^2\|\x-x_g^k\|^2.
\end{align*}
Thus,
\begin{align*}
    \mathbb{E}\big[2\langle\bar{x} - x_g^k, t_k\rangle \big| x^k, x_f^{k+1}\big] \leq& 2\left[r(\bar{x}) - r(x_f^{k+1})\right] - \mu \lVert x_f^{k+1} - \bar{x}\rVert^2
-\frac{1}{\theta}\lVert x_f^{k+1} - x_g^k\rVert^2 \\&- \theta\lVert \nabla r(x_f^{k+1})\rVert^2
+3\theta\lVert \nabla \bar{A}_\theta^k(x_f^{k+1})\rVert^2
\\&+3\theta \delta^2\lVert x_f^{k+1} - x_g^k\rVert^2
+3\theta \lVert \nabla p(x_g^k) - s_k \rVert ^ 2
\\=& 2\left[r(\bar{x}) - r(x_f^{k+1})\right] - \mu \lVert x_f^{k+1} - \bar{x}\rVert^2
\\&-\frac{1}{\theta}\left(1 - 3\theta^2\delta^2\right) \lVert x_f^{k+1} - x_g^k\rVert^2
- \theta\lVert \nabla r(x_f^{k+1})\rVert^2
\\&+3\theta\lVert \nabla \bar{A}_\theta^k(x_f^{k+1})\rVert^2
+3\theta \lVert \nabla p(x_g^k) - s_k \rVert ^ 2.
\end{align*}
With $\theta \leq \frac{1}{3\delta}$, we have
	\begin{align*}
	\mathbb{E}\big[2\langle\bar{x} - x_g^k, t_k\rangle \big| x^k, x_f^{k+1}\big]
	\leq& 2\left[r(\bar{x}) - r(x_f^{k+1})\right] - \mu \lVert x_f^{k+1} - \bar{x} \rVert^2
	-\frac{2}{3\theta} \lVert x_f^{k+1} - x_g^k \rVert^2 
	\\&- \theta \lVert \nabla r(x_f^{k+1}) \rVert^2 + 3\theta \lVert \nabla \bar{A}_\theta^k(x_f^{k+1}) \rVert^2 \\&+ 3 \theta \lVert \nabla p(x_g^k) - s_k \rVert ^ 2\\
	\leq& 2\left[r(\bar{x}) - r(x_f^{k+1})\right] - \mu \lVert x_f^{k+1} - \bar{x} \rVert^2 - \theta \lVert \nabla r(x_f^{k+1}) \rVert^2\\&+3\theta \lVert \nabla \bar{A}_\theta^k(x_f^{k+1}) \rVert^2 + 3 \theta \lVert \nabla p(x_g^k) - s_k \rVert ^ 2
	\\&-\frac{1}{3\theta} \lVert x_g^k - \arg\min_{x \in \mathbb{R}^d} \bar{A}_\theta^k(x) \rVert^2 \\
	&+\frac{2}{3\theta} \lVert x_f^{k+1} - \arg\min_{x \in \mathbb{R}^d} \bar{A}_\theta^k(x) \rVert^2.
\end{align*}
Here we used the fact that:
\begin{align*}
    -\|a-c\|^2\leq-\frac{1}{2}\|a-b\|^2+\|b-c\|^2.
\end{align*}
One can observe that $A_\theta^k(x)$ is $\frac{1}{\theta}$-strongly convex. Hence,
	\begin{align*}
	\mathbb{E}\big[2\langle\bar{x} - x_g^k, t_k\rangle \big| x^k, x_f^{k+1}\big]
	\leq&
	2\left[ r(\bar{x}) - r(x_f^{k+1}) \right] - \mu \lVert x_f^{k+1} - \bar{x} \rVert^2 - \theta \lVert \nabla r(x_f^{k+1}) \rVert^2
	\\&+3\theta \lVert \nabla \bar{A}_\theta^k(x_f^{k+1}) \rVert^2 + 3 \theta \lVert \nabla p(x_g^k) - s_k \rVert ^ 2
	\\&-\frac{1}{3\theta}\lVert x_g^k - \arg\min_{x \in \mathbb{R}^d} \bar{A}_\theta^k(x) \rVert^2
	+\frac{2\theta}{3}\lVert \nabla \bar{A}_\theta^k(x_f^{k+1}) \rVert^2.
\end{align*}
Now we are ready to construct the expression with the convergence criterion.
\begin{align*}
    \mathbb{E}\big[2\langle\bar{x} - x_g^k, t_k\rangle \big|& x^k, x_f^{k+1}\big]
    \\\leq& 2\left[ r(\bar{x}) - r(x_f^{k+1}) \right] - \mu \lVert x_f^{k+1} - \bar{x} \rVert^2
	- \theta \lVert \nabla r(x_f^{k+1}) \rVert^2
	\\&
	+\frac{11\theta}{3} \left( \lVert \nabla \bar{A}_\theta^k(x_f^{k+1}) \rVert^2 - \frac{1}{11\theta^2} \lVert x_g^k - \arg\min_{x \in \mathbb{R}^d} \bar{A}_\theta^k(x) \rVert^2 \right) \\&+ 3 \theta \lVert \nabla p(x_g^k) - s_k \rVert ^ 2
	\\
	=&\;
	2\left[ r(\bar{x}) - r(x_f^{k+1}) \right] - \mu \lVert x_f^{k+1} - \bar{x} \rVert^2
	- \theta \lVert \nabla r(x_f^{k+1}) \rVert^2
	\\&
	+\frac{11\theta}{3} \left( \lVert \nabla \bar{A}_\theta^k(x_f^{k+1}) \rVert^2 - \frac{9\delta^2}{11} \lVert x_g^k - \arg\min_{x \in \mathbb{R}^d} \bar{A}_\theta^k(x) \rVert^2 \right)
 \\& + 3 \theta \lVert \nabla p(x_g^k) - s_k \rVert ^ 2.
\end{align*}
This completes the proof of Lemma \ref{st_ae:lemma}.
\end{proof}

Let us move on to proof of Lemma \ref{st_ae:main_lemma}:
\begin{lemma} \textbf{(Lemma 1)}
    Under Assumptions \ref{ass:1}, \ref{ass:sim} and \ref{ass:2}, the condition (\ref{st_aux:grad_app}) and the tuning (\ref{st_ae:naive_choice}) we obtain:
        \begin{align*}\label{st:ae_rec}
	\mathbb{E} \bigg[ \frac{1}{\eta} & \lVert x^{k+1} - x^* \rVert^2
	+
	\frac{2}{\tau}\left[r(x_f^{k+1}) - r(x^*)\right]\bigg]
        \\
	\leq&
	\left(1-\alpha \eta +\frac{12 \theta\delta^2 \eta }{B}\cdot\zeta\right)  \cdot \mathbb{E}\left[ \frac{1}{\eta} \lVert x^k - x^* \rVert^2\right]+ (1-\tau) \cdot \E\left[\frac{2}{\tau}\left[r(x_f^k) - r(x^*)\right]\right] 
        \\
        &+\E\left[\left(\frac{12\theta\delta^2}{B\tau}-\frac{\mu}{6}\right) \cdot\zeta \left(\|x_f^k-x_*\|^2 + \|\x-x_*\|^2\right)  \right]
        + \frac{4\theta}{B\tau}\sigma^2.
        \end{align*}
\end{lemma}
\begin{proof}
Using Line \ref{st_ae:line:3}, we write
	\begin{align*}
		\frac{1}{\eta}\lVert x^{k+1} - x^* \rVert^2
		=&
		\frac{1}{\eta}\lVert x^k - x^* \rVert^2
		+\frac{2}{\eta}\langle x^{k+1} - x^k, x^k - x^* \rangle
		+\frac{1}{\eta}\lVert x^{k+1} - x^k \rVert^2
		\\=&
		\frac{1}{\eta}\lVert x^k - x^* \rVert^2
		+2\alpha\langle x_f^{k+1} -x^k, x^k - x^* \rangle
		-2\langle t_k, x^k - x^* \rangle
		\\&+\frac{1}{\eta}\lVert x^{k+1} - x^k \rVert^2
		\\=&
		\frac{1}{\eta}\lVert x^k - x^* \rVert^2
		+\alpha\lVert x_f^{k+1} - x^* \rVert^2
		-\alpha\lVert x^k - x^* \rVert^2
		\\&-\alpha\lVert x_f^{k+1} - x^k \rVert^2
		-2\langle t_k, x^k - x^* \rangle
		+\frac{1}{\eta}\lVert x^{k+1} - x^k \rVert^2,
	\end{align*}
Here we use the properties of the norm to reveal the inner product. Line \ref{st_ae:line:1} of Algorithm \ref{st_ae:alg} gives
	\begin{align*}
		\frac{1}{\eta}\lVert x^{k+1} - x^* \rVert^2
		=&
		\left(\frac{1}{\eta}-\alpha\right)\lVert x^k - x^* \rVert^2
		+\alpha\lVert x_f^{k+1} - x^* \rVert^2
		+\frac{1}{\eta}\lVert x^{k+1} - x^k \rVert^2
		\\&-\alpha\lVert x_f^{k+1} - x^k \rVert^2
		+2\langle t_k, x^* - x_g^k \rangle
		+\frac{2(1-\tau)}{\tau}\langle t_k, x_f^k - x_g^k \rangle.
	\end{align*}
Using \eqref{st_ae:eq:1} with $\bar{x} = x^*$ and $\bar{x} = x_f^k$, we get
\begin{align*}
	\mathbb{E}\left[\frac{1}{\eta}\lVert x^{k+1} - x^* \rVert^2\Bigg| x^k, x_f^{k+1}\right]
	\leq&
	\left(\frac{1}{\eta}-\alpha\right)\lVert x^k - x^* \rVert^2
	+\alpha\lVert x_f^{k+1} - x^* \rVert^2
	\\&+\mathbb{E}\left[\frac{1}{\eta}\lVert x^{k+1} - x^k \rVert^2\Bigg| x^k, x_f^{k+1}\right]
	-\alpha\lVert x_f^{k+1} - x^k \rVert^2
	\\&+2\left[r(x^*) - r(x_f^{k+1})\right] - \mu\lVert x_f^{k+1} - x^* \rVert^2
	\\&+\frac{2(1-\tau)}{\tau}\left[r(x_f^{k}) - r(x_f^{k+1})\right]
	-\frac{\theta}{\tau}\lVert \nabla r(x_f^{k+1}) \rVert^2 \\&- \mu\| \x - x_f^{k} \|^2 + \frac{3\theta}{\tau} \lVert \nabla p(x_g^k) - s_k \rVert ^ 2
 \\
 &+\frac{11\theta}{3\tau}\left(\lVert \nabla \bar{A}_\theta^k(x_f^{k+1}) \rVert^2 - \frac{9\delta^2}{11}\lVert x_g^k- \arg\min_{x \in \mathbb{R}^d} \bar{A}_\theta^k(x) \rVert^2\right).
\end{align*}
The last expression
\begin{align*}
    \frac{11\theta}{3\tau}\left(\lVert \nabla \bar{A}_\theta^k(x_f^{k+1}) \rVert^2 - \frac{9\delta^2}{11}\lVert x_g^k- \arg\min_{x \in \mathbb{R}^d} \bar{A}_\theta^k(x) \rVert^2\right),
\end{align*}
can be neglected due to the condition (\ref{st_aux:grad_app}).
Consider expectation separately:
\begin{align*}
    \mathbb{E}\left[\frac{1}{\eta}\lVert x^{k+1} - x^k \rVert^2\Bigg| x^k, x_f^{k+1}\right] =& \frac{1}{\eta}\mathbb{E}\left[\lVert \eta \alpha(x_f^{k+1} - x^k) - \eta t_k \rVert^2\Big| x^k, x_f^{k+1}\right]
     \\ \leq&
    3 \eta \alpha ^ 2 \lVert x_f^{k+1} - x^k \rVert^2 + 3 \eta\lVert \nabla r(x_f^{k + 1 }) \rVert^2 \\&+ 3 \eta \mathbb{E}\left[\lVert \nabla r(x_f^{k + 1 }) - t_k \rVert^2\Big| x^k, x_f^{k+1} \right].
\end{align*}
Here we added and subtracted $\eta\nabla r(\x)$ under the norm. After rearranging the expressions, we obtain
\begin{align*}
        \mathbb{E}\left[\frac{1}{\eta}\lVert x^{k+1} - x^* \rVert^2\Bigg| x^k, x_f^{k+1}\right]
	\leq&
	\left(\frac{1}{\eta}-\alpha\right)\lVert x^k - x^* \rVert^2
	+(\alpha - \frac{\mu}{3})\lVert x_f^{k+1} - x^* \rVert^2
	\\&+\alpha(3\eta\alpha - 1)\lVert x_f^{k+1} - x^k \rVert^2 + \frac{3\theta}{\tau} \lVert \nabla p(x_g^k) - s_k \rVert ^ 2
	\\&+\frac{2(1-\tau)}{\tau}\left[r(x_f^k) - r(x^*)\right]
	-\frac{2}{\tau}\left[r(x_f^{k+1}) - r(x^*)\right]
         \\&+ 3 \eta \lVert \nabla r(x_f^{k + 1 }) - t_k \rVert^2 - \frac{\mu}{3}\|\x-x_*\|^2
         \\&- \frac{\mu}{3}\|\x-x_*\|^2- \mu\|\x-x_f^k\|^2
        \\&+\left(3\eta - \frac{\theta}{\tau}\right)\lVert \nabla r(x_f^{k+1}) \rVert^2.
\end{align*}
The choice of $\alpha,\eta,\tau$ defined by (\ref{st_ae:naive_choice}) gives
	\begin{align*}
	\mathbb{E}\left[\frac{1}{\eta}\lVert x^{k+1} - x^* \rVert^2\Bigg| x^k, x_f^{k+1}\right]
	\leq&
	\left(\frac{1}{\eta}-\alpha\right)\lVert x^k - x^* \rVert^2
	+\frac{2(1-\tau)}{\tau}\left[r(x_f^k) - r(x^*)\right]
	\\&-\frac{2}{\tau}\left[r(x_f^{k+1}) - r(x^*)\right] + \frac{3\theta}{\tau} \lVert \nabla p(x_g^k) - s_k \rVert ^ 2
         \\&+ \frac{\theta}{\tau}\lVert \nabla r(x_f^{k + 1 }) - t_k \rVert^2- \frac{\mu}{3}\|\x-x_*\|^2
        \\&- \frac{\mu}{3}\|\x-x_*\|^2 - \mu\|\x-x_f^k\|^2.
\end{align*}
Note that
\begin{align*}
    -\frac{\mu}{3}\|\x-x_*\|^2\leq -\frac{\mu}{6}\|x_f^k-x_*\|^2 + \frac{\mu}{3}\|\x-x_f^k\|^2.
\end{align*}
Here we need a technical lemma:
\begin{lemma}\label{lemma:stoch_oracles}
    Let $p(x)=r(x)-r_1(x)$ and $s_k=\frac{1}{B}\sum_{i=1}^B(\nabla r_{m_i}(x_g^k, \xi)-\nabla r_1(x_g^k))$. Then:
    \begin{align*}
        \E_{\xi,i}\left[\|s_k-\nabla p(x_g^k)\|^2\right] \leq \frac{2(\sigma_{sim}^2+\sigma_{noise}^2)}{B}+ \frac{4\delta^2}{B}\|\x-x_*\|^2\cdot\zeta.
    \end{align*}
\end{lemma}
\begin{proof}
    Let us apply the Young's inequality
    \begin{align*}
        \E_{\xi,i}[\|s_k-\nabla p(x_g^k)\|^2] \\=& 
        \E_{\xi,i}\left[\left\|\nabla r(x_g^k) - \frac{1}{B}\sum_{i=1}^B\nabla r_{m_i}(x_g^k,\xi_{i}) \pm \frac{1}{B}\sum_{i=1}^B\nabla r_{m_i}(x_g^k)\right\|^2\right]
        \\
        \leq& 2\E_{i}\left[\left\| \nabla r(x_g^k)- \frac{1}{B}\sum_{i=1}^B\nabla r_{m_i}(x_g^k)\right\|^2\right] \\&+ \frac{2}{B^2}\E_{i,\xi}\left\| \sum_{i=1}^B (\nabla r_{m_i}(x_g^k)-\nabla r_{m_i}(x_g^k,\xi_i)) \right\|^2.
    \end{align*}
    Next, we look at the expression, given by the second term:
    \begin{align*}
        &\frac{2}{B^2}\sum_{i=1}^B\|\nabla r_{m_i}(x_g^k)-\nabla r_{m_i}(x_g^k,\xi_i)\|^2 \\&+ \frac{2}{B^2}\sum_{i\neq j=1}^B\langle \nabla r_{m_i}(x_g^k)-\nabla r_{m_i}(x_g^k,\xi_i), \nabla r_{m_j}(x_g^k)-\nabla r_{m_j}(x_g^k,\xi_j) \rangle.
    \end{align*}
    Let us use the tower property and take the $\xi$ expectation first. Since $m_i$ and $m_j$ are independent, we can enter the expectation into the multipliers of the second summand separately. Then only the first summand will remain, which is estimated based on Assumption \ref{ass:sim}. Using Assumption \ref{ass:sim}, we obtain
    \begin{align*}
        \E_{\xi,i}\left[\|s_k-\nabla p(x_g^k)\|^2\right]\leq \frac{2(\sigma_{sim}^2+\sigma_{noise}^2)}{B} + \frac{4\delta^2}{B}\|\x-x_*\|^2\cdot\zeta.
    \end{align*}
\end{proof}
Let us return to the proof of the lemma. Denote $\sigma^2=2(\sigma_{sim}^2+\sigma_{noise}^2)$. Taking expectation and with \eqref{st_aux:grad_app}, we have
	\begin{align*}
	\mathbb{E}\bigg[\frac{1}{\eta}&\lVert x^{k+1} - x^* \rVert^2
	+
	\frac{2}{\tau}\left[r(x_f^{k+1}) - r(x^*)\right]\bigg]
	\\\leq&
	\mathbb{E}\left[\left(\frac{1}{\eta}-\alpha+\frac{12\theta\delta^2}{B}\cdot\zeta\right)\lVert x^k - x^* \rVert^2\right] \\&+\E\left[\frac{2(1-\tau)}{\tau}\left[r(x_f^k) - r(x^*)\right]\right] 
        \\
        &+\E\left[\left(\frac{12\theta\delta^2}{B\tau}-\frac{\mu}{6}\right)\|x_f^k-x_*\|^2\cdot\zeta\right] \\&+\E\left[\left(\frac{4\theta\delta^2}{B\tau}-\frac{\mu}{3}\right)\|\x-x_*\|^2\cdot\zeta\right]
        \\
        &+ \frac{4\theta}{B\tau}\sigma^2.
        \end{align*}
\end{proof}

\section{Theorem 2}\label{ap:th_set2}
\begin{theorem}(\textbf{Theorem 2})
    In Lemma \ref{st_ae:main_lemma} choose $\zeta=1$ and the tuning (\ref{st_ae:choice}).
    Then we have:
    \begin{align*}
        \E[\Phi_{k+1}]\leq\left( 1-\frac{\sqrt{\mu\theta}}{18} \right)\E[\Phi_k] + \frac{4\theta}{B},
    \end{align*}
    where $\Phi_k = \frac{\tau}{\eta}\|x^k-x^*\|^2 + 2[r(x_f^k)-r(x^*)]$.    
\end{theorem}
\begin{proof}
    We consider possible values of $\eta$ one by one.\\
    $\bullet$ Consider $\eta=\frac{1}{\mu}$. In this case, we have:
    \begin{align*}
        1&-\eta\alpha+\frac{12\eta\theta\delta^2}{B}\leq\frac{2}{3} + \frac{12\delta^2}{B\mu}\frac{\mu B\tau}{72\delta^2}\leq\frac{5}{6}=1-\frac{1}{6}.
    \end{align*}
    $\bullet$ Consider $\eta-\frac{\theta}{3\tau}$: 
    \begin{align*}
        1-\eta\alpha+\frac{12\eta\theta\delta^2}{B}=&1-\frac{\mu\theta}{9\tau}+\frac{4\theta^2\delta^2}{B\tau}\leq 1-\frac{\mu\theta}{9\tau} +\frac{4\theta\delta^2}{B\tau}\frac{\mu B\tau}{72\delta^2} \leq 1-\frac{\mu\theta}{18\tau}
        \\
        =& 1-\frac{\sqrt{\mu\theta}}{18}. 
    \end{align*}
    Thus, we obtain:
    \begin{align*}
        \E[\Phi_{k+1}]\leq\max\left\{ 1-\frac{1}{6},1-\frac{\sqrt{\mu\theta}}{18} \right\}.
    \end{align*}
    Since $\mu<\delta$, we obtain the required.
\end{proof}

\section{Corrolary 1, 2}\label{ap:proof_cor_1}
Here we introduce a proof for Corollary \ref{st_ae:cor_1} and Corollary \ref{st_ae:cor_2}. We use tecnhique from \citep{opt_sgd}.
\begin{proof}
    Earlier we obtained the recurrence formula:
    \begin{align*}
        0\leq (1-a\sqrt{\theta})\E[\Phi_{k}]-\E[\Phi_{k+1}] + c\theta. 
    \end{align*}
    Let us multiply both sides by weight $w_k=(1-a\sqrt{\theta})^{-(k+1)}$ and rewrite:
    \begin{align*}
        0\leq \frac{1}{\sqrt{\theta}}\left(\E[\Phi_{k}]w_{k-1}-\E[\Phi_{k+1}]w_k\right) + cw_k\sqrt{\theta}.
    \end{align*}
    Denote $W_N =\sum_{k=0}^Nw_k(1-\sqrt{\theta})^{-(N+1)}\sum_{k=0}^N(1-a\sqrt{\theta})^k\leq\frac{w_N}{a\sqrt{\theta}}$ and sum up with $k=0,...,N$:
    \begin{align*}
        \frac{w_N\E[\Phi_{k+1}]}{\sqrt{\theta}W_N}\leq \frac{\Phi_{0}}{\sqrt{\theta}W_N}+c\sqrt{\theta}.
    \end{align*}
    Let us rewrite:
    \begin{align*}
        \E[\Phi_{k+1}]\leq \frac{1}{a\sqrt{\theta}}\exp\{-a\sqrt{\theta}(N+1)\} + \frac{c\sqrt{\theta}}{a}.
    \end{align*}
    To complete the proof we need to substitute $a=\frac{\sqrt{\mu}}{3}$ and $c=\frac{4\sigma^2}{B}$ for Corollary \ref{st_ae:cor_1} and $a=\frac{\sqrt{\mu}}{18}$ and $c=\frac{4\sigma^2}{B}$ for Corollary \ref{st_ae:cor_2} and look at different $\theta$.\\
    For Corollary \ref{st_ae:cor_1} we obtain:\\
    If $ \frac{9\ln^2(\max\{2, \frac{B\mu N^2 \Phi_0}{36\sigma^2}\})}{\mu N^2} \leq\frac{1}{3\delta},$ where $\sigma^2=2(\sigma_{sim}^2+\sigma_{noise}^2)$, then choose $\theta=\frac{9\ln^2(\max\{2, \frac{B\mu N^2 \Phi_0}{36\sigma^2}\})}{\mu N^2}$ and get:
    \begin{align*}
        \E[\Phi_{N+1}] = \mathcal{O}\left( \frac{24(\sigma_{sim}^2+\sigma_{noise}^2)}{\mu BN}\right).
    \end{align*}
    Otherwise, choose $\theta = \frac{1}{3\delta}$ and get:
    \begin{align*}
        \E[\Phi_{N+1}] = \mathcal{O}\left( 
        \sqrt{3\delta}\Phi_0\exp\left\{ -\frac{\sqrt{\mu}}{3\sqrt{3}\sqrt{\delta}}N \right\} +\frac{24(\sigma_{sim}^2+\sigma_{noise}^2)}{\mu BN} \right).
    \end{align*}
    It remains to remember that method iteration requires $\nicefrac{B}{M}$ communications.
    For Corollary \ref{st_ae:cor_2} we obtain:\\
    $\bullet$ If $\frac{1}{3\delta}\leq\frac{\mu^3B^2}{5184\delta^4}$ and $ \frac{324\ln^2(\max\{2, \frac{B\mu N^2 \Phi_0}{1296\sigma^2}\})}{\mu N^2} \leq\frac{1}{3\delta},$ where $\sigma^2=2(\sigma_{sim}^2+\sigma_{noise}^2)$, then choose $\theta=\frac{324\ln^2(\max\{2, \frac{B\mu N^2 \Phi_0}{1296\sigma^2}\})}{\mu N^2}$ and get:
    \begin{align*}
        \E[\Phi_{N+1}] = O\left( \frac{144(\sigma_{sim}^2+\sigma_{noise}^2)}{\mu BN}\right).
    \end{align*}
    $\bullet$ If $\frac{1}{3\delta}\leq\frac{\mu^3B^2}{5184\delta^4}$ and $ \frac{324\ln^2(\max\{2, \frac{B\mu N^2 \Phi_0}{1296\sigma^2}\})}{\mu N^2} \geq\frac{1}{3\delta},$ where $\sigma^2=2(\sigma_{sim}^2+\sigma_{noise}^2)$, then choose $\theta=\frac{1}{3\delta}$ and get:
    \begin{align*}
        \E[\Phi_{N+1}] = \mathcal{O}\left( 
        \sqrt{3\delta}\Phi_0\exp\left\{ -\frac{\sqrt{\mu}}{18\sqrt{3}\sqrt{\delta}}N \right\} +\frac{144(\sigma_{sim}^2+\sigma_{noise}^2)}{\mu BN} \right).
    \end{align*}
    $\bullet$ If $\frac{1}{3\delta}\geq\frac{\mu^3B^2}{5184\delta^4}$ and $ \frac{324\ln^2(\max\{2, \frac{B\mu N^2 \Phi_0}{1296\sigma^2}\})}{\mu N^2} \leq\frac{\mu^3B^2}{5184\delta^4},$ where $\sigma^2=2(\sigma_{sim}^2+\sigma_{noise}^2)$, then choose $\theta=\frac{324\ln^2(\max\{2, \frac{B\mu N^2 \Phi_0}{1296\sigma^2}\})}{\mu N^2}$ and get:
    \begin{align*}
        \E[\Phi_{N+1}] = \mathcal{O}\left( \frac{144(\sigma_{sim}^2+\sigma_{noise}^2)}{\mu BN}\right).
    \end{align*}
    $\bullet$ If $\frac{1}{3\delta}\geq\frac{\mu^3B^2}{5184\delta^4}$ and $ \frac{324\ln^2(\max\{2, \frac{B\mu N^2 \Phi_0}{1296\sigma^2}\})}{\mu N^2} \geq\frac{\mu^3B^2}{5184\delta^4},$ where $\sigma^2=2(\sigma_{sim}^2+\sigma_{noise}^2)$, then choose $\theta=\frac{324\ln^2(\max\{2, \frac{B\mu N^2 \Phi_0}{1296\sigma^2}\})}{\mu N^2}$ and get:
    \begin{align*}
        \E[\Phi_{N+1}] = \mathcal{O}\left( 
        \frac{72\delta^2}{B\sqrt{\mu^3}}\Phi_0\exp\left\{ -\frac{B}{1296}\frac{\mu^2}{\delta^2}N \right\} +\frac{144(\sigma_{sim}^2+\sigma_{noise}^2)}{\mu BN} \right).
    \end{align*}
    As in the previous paragraph, the method iteration requires $\nicefrac{B}{M}$ communications.
\end{proof}

\section{\texttt{ASEG} for the Convex Case}\label{ap:aseg_conv}
The next algorithm is the adaptation of Algorithm \ref{st_ae:alg} for the convex case. We use time-varying $\tau_k$ and $\eta_k$.
\begin{algorithm}[H]
	\caption{\texttt{Accelerated Stochastic ExtraGradient} for convex case}
        \label{st_ae:alg_conv}
	\begin{algorithmic}[1]
		\State {\bf Input:} $x^0=x_f^0 \in \mathbb{R}^d$
		\State {\bf Parameters:} $\tau \in (0,1]$, $\eta,\theta,\alpha > 0, N \in \{1,2,\ldots\}, B \in \mathbb{N}$
		\For{$k=0,1,2,\ldots, N-1$}:
            \For{server}:
                \State Generate set $|I|=B$ numbers uniformly from $[2,...,n]$
                \State Send to each device $m$ one bit $b_k^m$: $1$ if $m\in I$, 0 otherwise
            \EndFor
            \For{each device $m$ in parallel}:
			\State $x_g^k = \tau_{k+1} x^k + (1-\tau_{k+1}) x_f^k$\label{conv_ae:line:1}
                \If{$b_k^m=1$}
                    \State Send $\nabla r_m(x_g^k, \xi)$ to server
                \EndIf
            \EndFor
            \For{server}:
                \State $s_k=\frac{1}{B}\sum_{m\in I}(\nabla r_m(x_g^k, \xi)-\nabla r_1(x_g^k))$
    		\State $x_f^{k+1} \approx \arg\min_{x \in \mathbb{R}^d}\left[ \A(x) \coloneqq \langle s_k,x - x_g^k\rangle + \frac{1}{2\theta}\lVert x - x_g^k \rVert^2 + r_1(x)\right]$\label{conv_ae:line:2}
                \State Generate set $|I|=B$ numbers uniformly from $[1,...,n]$
                \State Send to each device $m$ one bit $c_k^m$: $1$ if $m\in I$, 0 otherwise
            \EndFor
            \For{each device $m$ in parallel}:
                \If{$c_k^m=1$}
                    \State Send $\nabla r_m(x_f^{k+1}, \xi)$ to server
                \EndIf
            \EndFor
            \For{server}:
                \State $t_k = \frac{1}{B}\sum_{m\in I}\nabla r_m(x_f^{k+1}, \xi)$
			\State $x^{k+1} = x^k - \eta_{k+1} t_k$\label{conv_ae:line:3}
            \EndFor
		\EndFor
		\State {\bf Output:} $x^N$
	\end{algorithmic}
\end{algorithm}

\begin{theorem}\label{st_ae:theorem3}
    Consider Algorithm \ref{st_ae:alg_conv} to run $N$ iterations under Assumptions \ref{ass:1}, \ref{ass:2} and \ref{ass:sim} with $\zeta=0$ and $\mu=0$, with the following tuning:
    \begin{equation*}
        \tau_k=\frac{2}{k+1}, \quad \theta\leq\frac{1}{3\delta},\quad \eta_k=\frac{\theta}{\tau_k},
    \end{equation*}
    and let $\x$ in Line \ref{conv_ae:line:2} satisfy
    \begin{equation}
        \lVert\nabla \bar{A}_\theta^k(x_f^{k+1})\rVert^2 \leq  \frac{9\delta^2}{11}\lVert x_g^k- \arg\min_{x \in \mathbb{R}^d} \bar{A}_\theta^k(x)\rVert^2.
    \end{equation}
    Then the following inequality holds:
    \begin{equation*}
        \E\left[\frac{\theta}{\tau_{N+1}^2}[r(x_f^{N+1})-r(x_*)]\right] \leq \lVert x_0-x_*\rVert^2 + \frac{6(\sigma_{sim}^2+\sigma_{noise}^2)}{B}\theta^2\sum_{i=1}^N\frac{1}{\tau_{k}^2}.
    \end{equation*}
\end{theorem}

\begin{proof}
    Let us denote $\sigma^2=2(\sigma_{sim}^2+\sigma_{noise}^2)$ again. Here we apply Lemma \ref{st_ae:lemma} with $\mu=0$ and $\zeta=0$.
    \begin{align*}
        \frac{1}{\eta_{k+1}}\E\big[\lVert x_{k+1}-x_* \rVert^2\big|&x_k\big] \\\leq& \frac{1}{\eta_{k+1}}\lVert x_k-x_*\rVert^2 + \frac{1}{\eta_{k+1}}\E\left[\lVert x_{k+1}-x_k \rVert^2\big|x_k\right] \\&+ 2[r(x_*)-r(\x)]
        + \frac{2(1-\tau_{k+1})}{\tau_{k+1}}[r(x_f^k)-r(x_*)] \\&- \frac{\theta}{\tau_{k+1}}\lVert\nabla r(\x)\rVert^2 + \frac{3\theta}{B\tau_{k+1}}\sigma^2
        \\
        \leq& \frac{1}{\eta_{k+1}}\lVert x_k-x_*\rVert^2 + \frac{\eta_{k+1}}{B}\sigma^2 + \frac{3\theta}{B\tau_{k+1}}\sigma^2 \\&+ \left(\eta_{k+1} - \frac{\theta}{\tau_{k+1}}\right) \lVert\nabla r(\x)\rVert^2
        + \frac{2}{\tau_{k+1}}\left[r(x_*)-r(\x)\right] \\&+ \frac{2(1-\tau_{k+1})}{\tau_{k+1}}\left[r(x_f^k)-r(x_*)\right].
    \end{align*}
    Let us substitute $\eta_{k+1}=\frac{\theta}{2\tau_{k+1}}$:
    \begin{align*}
        \frac{\tau_{k+1}}{\theta}\E\big[\lVert x_{k+1}-x_* \rVert^2\big|&x_k\big] \\\leq& \frac{\tau_{k+1}}{\theta}\lVert x_k-x_*\rVert^2 + \frac{\theta}{4B\tau_{k+1}}\sigma^2 + \frac{3\theta}{2B\tau_{k+1}}\sigma^2
        \\&+ \frac{1}{\tau_{k+1}}\left[r(x_*)-r(\x)\right] + \frac{1-\tau_{k+1}}{\tau_{k+1}}\left[r(x_f^k)-r(x_*)\right].
    \end{align*}
    Multiply both sides by $\frac{\theta}{\tau_{k+1}}$:
    \begin{align*}
        \E\big[\lVert x_{k+1}-x_* \rVert^2\big|&x_k\big] \\\leq& \lVert x_k-x_*\rVert^2 + \frac{3\theta^2}{B\tau_{k+1}^2}\sigma^2 + \frac{\theta}{\tau_{k+1}^2}[r(x_*)-r(\x)] \\&+ \frac{\theta(1-\tau_{k+1})}{\tau_{k+1}^2}\left[r(x_f^k)-r(x_*)\right].
    \end{align*}
    Let us rewrite the last inequality:
    \begin{align*}
        \E\big[\lVert x_{k+1}-x_* \rVert^2\big|&x_k\big] + \frac{\theta}{\tau_{k+1}^2}[r(\x)-r(x_*)] \\\leq&
        \lVert x_k-x_*\rVert^2 + \frac{3\theta^2}{B\tau_{k+1}^2}\sigma^2 \\&+ \frac{\theta(1-\tau_{k+1})}{\tau_{k+1}^2}\left[r(x_f^k)-r(x_*)\right].
    \end{align*}
    Define $\Phi_k=\lVert x_{k}-x_* \rVert^2 + \frac{\theta}{\tau_{k}^2}[r(x_f^k)-r(x_*)]$. Using $\tau_k=\frac{2}{k+1}$ we obtain:
    \begin{align*}
        \frac{\theta}{\tau_{k+1}^2}[r(\x)-r(x_*)] \leq \Phi_k + \frac{3\theta^2}{B\tau_{k+1}^2}\sigma^2.
    \end{align*}
    Next, we apply previous inequality and assume $\tau_1=1$:
    \begin{align*}
        \E\left[\frac{\theta}{\tau_{N+1}^2}[r(x_f^{N+1})-r(x_*)]\right] \leq& \lVert x_0-x_*\rVert^2 + \frac{3\sigma^2\theta^2}{B}\sum_{i=1}^N\frac{1}{\tau_{k}^2}.
    \end{align*}
\end{proof}

Similarly to the $\mu>0$ case, it is not possible to ensure convergence by a naive choice of $\theta$. We propose a choice of $\theta$ that will ensure optimal convergence.

\begin{corollary}
         Consider Assumptions \ref{ass:1} with $\mu=0$, \ref{ass:sim} with $\zeta=0$ and \ref{ass:2}. Let Algorithm \ref{st_ae:alg_conv} run for $N$ iterations with tuning of $\tau_k,\eta_k$ as in Theorem \ref{st_ae:theorem3}. We propose the following tuning of $\theta$:\\
    If $\frac{\sqrt{B}\|x_0-x_*\|}{\sqrt{3}\sigma(N+1)^{\nicefrac{3}{2}}} \leq \frac{1}{3\delta},$ then choose $\theta=\frac{\sqrt{B}\|x_0-x_*\|}{\sqrt{3}\sigma(N+1)^{\nicefrac{3}{2}}}$ and get:
    \begin{align*}
        \E\left[ r(x_f^{N+1})-r(x_*) \right] \leq \frac{2\sqrt{3}\sigma\|x_0-x_*\|}{\sqrt{B}\sqrt{N+1}}.
    \end{align*}
    Otherwise, choose $\theta = \frac{1}{3\delta}$ and get:
    \begin{align*}
        \E\left[ r(x_f^{N+1})-r(x_*) \right] \leq \frac{3\delta\lVert x_0-x_*\rVert^2}{(N+1)^2} + \frac{\sqrt{3}\sigma\|x_0-x_*\|}{\sqrt{B}\sqrt{N+1}}.
    \end{align*}
\end{corollary}
\begin{proof}
       Note that $\frac{1}{\tau_{k}^2}\leq \frac{(N+1)^2}{4}$. Thus, using Theorem \ref{st_ae:theorem3}, we obtain:
    \begin{align*}
        \E\left[ \theta(N+1)^2[r(x_f^{N+1})-r(x_*)] \right] \leq \lVert x_0-x_*\rVert^2 + \frac{3\sigma^2}{B}\theta^2(N+1)^3.
    \end{align*}
    Divide both sides by $\theta(N+1)^2$:
    \begin{align*}
        \E\left[ r(x_f^{N+1})-r(x_*) \right] \leq \frac{\lVert x_0-x_*\rVert^2}{\theta(N+1)^2} + \frac{3\sigma^2}{B}\theta(N+1).
    \end{align*}
    Let us notice that $\theta=\frac{\sqrt{B}\|x_0-x_*\|}{\sqrt{3}\sigma(N+1)^{\nicefrac{3}{2}}}$ equalizes the terms of the right part. If $\frac{\sqrt{B}\|x_0-x_*\|}{\sqrt{3}\sigma(N+1)^{\nicefrac{3}{2}}} \leq \frac{1}{3\delta}$, choose $\theta=\frac{\sqrt{B}\|x_0-x_*\|}{\sqrt{3}\sigma(N+1)^{\nicefrac{3}{2}}}$. Then we obtain:
    \begin{align*}
        \E\left[ r(x_f^{N+1})-r(x_*) \right] \leq \frac{2\sqrt{3}\sigma\|x_0-x_*\|}{\sqrt{B}\sqrt{N+1}}.
    \end{align*}
    Otherwise, choose $\theta=\frac{1}{3\delta}$ and obtain:
    \begin{align*}
        \E\left[ r(x_f^{N+1})-r(x_*) \right] \leq \frac{3\delta\lVert x_0-x_*\rVert^2}{(N+1)^2} + \frac{\sqrt{3}\sigma\|x_0-x_*\|}{\sqrt{B}\sqrt{N+1}}.
    \end{align*}
\end{proof}

\end{document}